%% file: alveare_final3.tex
\documentclass[11pt,a4paper,reqno]{amsart}
\usepackage{vmargin,color}

\usepackage[utf8]{inputenc}
\usepackage{amssymb,amsfonts}
\usepackage[all,arc]{xy}
\usepackage{color}
\usepackage{amsmath,amsfonts,amsthm,epsfig}

\usepackage[colorlinks=true,linkcolor=blue,citecolor=blue]{hyperref}

\def\ll{[\![}
\def\rr{]\!]}

\def\R{\mathbb R}
\def\N{\mathbb N}
\def\C{\mathcal C}

\def\cal{\mathcal}

\def\E{{\cal E}}
\def\F{{\cal F}}
\def\G{{\cal G}}
\def\H{{\cal H}}

\def\P{{\cal P}}

\def\PP{{\rm Pol}}

\def\Lip{{\rm Lip}}
\def\tt{\mathbf{t}}

\def\Cir{{\rm Cir}}

\def\a{\alpha}
\def\b{\beta}
\def\g{\gamma}
\def\de{\delta}
\def\e{\varepsilon}

\def\om{\omega}
\def\s{\sigma}

\def\t{\mathbf{t}}

\def\G{\Gamma}
\def\Om{\Omega}

\def\pa{\partial}

\def\d{{\rm d}\,}
\def\hd{{\rm hd}\,}

\def\arc{{\rm arc}}

\DeclareMathOperator*{\diam}{diam}

\DeclareMathOperator*{\weakstar}{\overset{\ast}{\rightharpoonup}}

\def\big{\bigskip}

\newtheorem{theorem}{Theorem}[section]
\newtheorem*{theorem*}{Theorem}

\newtheorem*{corollary*}{Corollary}

\theoremstyle{definition}

\newtheorem{remark}[theorem]{Remark}

\newtheorem{theoremletter}{Theorem}

\numberwithin{equation}{section}
\numberwithin{figure}{section}

\title[]{On the emergence of almost-honeycomb structures in low-energy planar clusters}

\author{M. Caroccia}
\address{Politecnico di Milano, Dipartimento di matematica, Via Edoardo Bonardi 9 - 20133 Milano, Italy}
\email{marco.caroccia@polimi.it}

\author{K. DeMason}
\address{Department of Mathematics, The University of Texas at Austin,  2515 Speedway Stop C1200, Austin, Texas 78712-1202, USA}
\email{kdemason@utexas.edu}

\author{F. Maggi}
\address{Department of Mathematics, The University of Texas at Austin,  2515 Speedway Stop C1200, Austin, Texas 78712-1202, USA}
\email{maggi@math.utexas.edu}

\begin{document}

\begin{abstract}
Several commonly observed physical and biological systems are arranged in shapes that closely resemble an honeycomb cluster, that is, a tessellation of the plane by regular hexagons. Although these shapes are not always the direct product of energy minimization, they can still be understood, at least phenomenologically, as low-energy configurations. In this paper, explicit quantitative estimates on the geometry of such low-energy configurations are provided, showing in particular that the vast majority of the chambers must be generalized polygons with six edges, and be closely resembling regular hexagons. Part of our arguments is a detailed revision of the estimates behind the global isoperimetric principle for honeycomb clusters due to Hales \cite{hales}. \end{abstract}

\maketitle
\tableofcontents

\section{Introduction} \subsection{Overview} Honeycomb-like structures are commonly observed in physical and biological systems. While in some cases (e.g., in foams) their emergence is consequence of an {\it energy minimization} property, in many other systems, like those resulting from geological mechanisms (e.g., basalt columns), animal behavior (e.g., beehives) or animal morphology (e.g., compound eyes of arthropods), what is observed results from complex processes leading to the formation of {\it low-energy} (rather than energy minimizing) configurations. This fact provides the motivation to extend and adapt Hales' isoperimetric theorem \cite{hales}, which characterizes {\bf honeycomb clusters} (i.e., tessellations of the plane by unit-area, regular hexagons) as the only {\it energy minimizing} configurations {\it under periodic boundary conditions}, to the case of low-energy configurations (and without the periodicity restriction).
In this paper, partly by a comprehensive review of Hales' methods, we undertake the analysis of such low-energy configurations. In rough terms, our main result (Theorem \ref{thm main} below) states that {\it every low-energy planar cluster with $N$-many unit-area, simply connected chambers must be close, in a quantitative way as $N\to\infty$, to an ideal honeycomb, with a controlled number of defects and with a vast preponderance of almost-hexagonal cells}; and this, without said cluster being required to possess any energy minimizing property.

\subsection{Isoperimetric clusters and Hales' theorem} We frame our work in the setting of Almgren's theory of isoperimetric clusters \cite{Almgren76} as presented in \cite[Part IV]{maggiBOOK}. Since we only consider {\it planar} clusters, we do not need to work with Borel sets and sets of finite perimeter, and work directly with open sets with Lipschitz boundary $E\subset\R^2$. We denote by $|E|$ the area (Lebesgue measure) of any such $E$, and by $P(E)=\H^1(\pa E)$ its perimeter (one-dimensional Hausdorff measure of its topological boundary).

\medskip

Given $N\in\N$, $N\ge 2$, a planar {\bf $N$-cluster} is a family $\E=\{\E(h)\}_{h=1}^N$ of mutually disjoint, non-empty open sets $\E(h)\subset\R^2$ (the {\bf chambers of $\E$}) with finite area and Lipschitz regular boundary. The {\bf area and perimeter} of an $N$-cluster $\E$ are then
\[
|\E|=\left(|\E(1)|,...,|\E(N)|\right)\in\R^N\,,\qquad P(\E)=\frac12\sum_{h=0}^NP(\E(h))=\H^1(\pa\E)\,,
\]
where $\E(0)=\R^2\setminus\bigcup_{h=1}^N\overline{\E(h)}$ denotes the {\bf exterior chamber} of $\E$ (so that $|\E(0)|=+\infty$), and where $\pa\E=\bigcup_{h=1}^N\pa\E(h)$ is the {\bf boundary} of $\E$. We say that $\E$ is an {\bf isoperimetric cluster} if
\[
P(\E)\le P(\F)\,,\qquad\mbox{for all clusters $\F$ with  $|\F|=|\E|$}\,;
\]
and a {\bf unit-area isoperimetric cluster} if, in addition, $|\E(h)|=1$ for all $h=1,...,N$.

\medskip

Notice that the chambers of an $N$-cluster are not assumed to be connected. A disjoint family of $N'$-many open, connected subsets of $\R^2$ with Lipschitz boundaries could thus be regarded (in more than one way) as an $N$-cluster for different values of $N\le N'$. From a conceptual viewpoint, this is clearly a somehow flawed feature of the notion of $N$-cluster used here. However, this feature is also crucially important for proving Almgren's existence theorem \cite{Almgren76} for isoperimetric clusters: {\it for every $v\in\R^N$ with positive coordinates, there exist isoperimetric $N$-clusters $\E$ with $|\E|=v$; see \cite[Chapter 29]{maggiBOOK}.}

\medskip

This said, we have the following natural {\bf connectedness conjecture} about isoperimetric clusters:
\begin{equation}
  \label{conjecture simply connected}\tag{CC}
  \mbox{{\it isoperimetric clusters have simply connected chambers}.}
\end{equation}
This conjecture is widely open. Its validity has been confirmed in the few cases ($N=2$ \cite{foisyzimba}, $N=3$ \cite{wichi}, and $N=4$ with equal areas \cite{paolinitamagnini,paolinitortorelli}) where a complete classification of isoperimetric clusters is known. One does not expect to carry this approach much further, as obtaining characterizations of isoperimetric clusters with arbitrary values of $N$ and $v$ seems out of question. The two basic facts about isoperimetric clusters that are known for every $N$ and $v$ are the validity of {\bf Plateau laws} (boundaries of isoperimetric clusters consist of finitely many circular arcs/segments meeting in threes at 120-degrees; see, e.g. \cite[Theorem 30.7]{maggiBOOK}), and the local finiteness and constancy in $v$ of the possible diffeomorphic types of $\pa\E$; see \cite[Theorem 1.9]{CLM1}.

\medskip

In his celebrated paper \cite{hales}, Hales presents an argument that serves to prove two isoperimetric principles concerning ``honeycombs''. The first one states that, if $a,b>0$ are such that the flat torus $\R^2[a,b]$ of width $a$ and height $b$ admits a tiling $\H$ by $N$ unit-area, regular hexagons, then $\H$ is the unique unit-area isoperimetric $N$-cluster (modulo translations); see \cite{carocciamaggi} for a quantitative analysis of this isoperimetric principle. The second one states that if $\E$ is a unit-area isoperimetric $N$-cluster in $\R^2$, then\footnote{By a slight refinement of Hales' argument, \eqref{hales lower bound} can actually be improved to $\psi(N)\ge (12)^{1/4}\,N+K_0\,\sqrt{N}$ with $K_0=\sqrt\pi - \sqrt[4]{3}/\sqrt{2}\approx 0.84$. Such refinement simply consists in applying ``Dido's inequality'' to quantify the size of some non-negative terms that were just discarded in Hales' original presentation; see the proof of \cite[Theorem 2.1]{heppesmorgan}.}
\begin{equation}
\label{hales lower bound}
\psi(N)>(12)^{1/4}\,N\,,\qquad\forall N\ge 2\,,
\end{equation}
where we have set
\begin{eqnarray*}
&&\psi(N)=P(\E_N)\,,
\\
&&\mbox{($\E_N$ a generic unit-area isoperimetric $N$-cluster)}\,.
\end{eqnarray*}
The energy bound \eqref{hales lower bound} is the only evidence towards a second fascinating, challenging, and largely unexplored {\bf honeycomb conjecture}:
\begin{equation}\label{honeycomb conjecture}\tag{HC}
  \begin{split}
  &\mbox{{\it unit-area isoperimetric $N$-clusters with $N$ large}}
  \\
  &\mbox{{\it should locally coincide with honeycombs}}\,.
  \end{split}
\end{equation}
To understand the connection between \eqref{hales lower bound} and the honeycomb geometry, we notice that, starting from a unit-area regular hexagon $H$, one can add a first complete layer of unit-area regular hexagons around $H$ (which, when complete, results in a cluster of $N=7$ unit-area regular hexagons), then a second complete layer (which, when complete, results in a cluster of $N=19$ unit-area regular hexagons), and so on. For every $N\ge2$ one can smooth out the exterior edges of such clusters into circular arcs. The resulting construction, detailed in \cite[Theorem 2.1]{heppesmorgan}, gives
\begin{equation}\label{the energy upper bound}
\psi(N)\le (12)^{1/4}\,N+M_0\,\sqrt{N}+3\,,\qquad\forall N\ge 2\,,
\end{equation}
(where $M_0=\pi/A_0^{1/2}\approx 1.95$ and $A_0=3\sqrt{3}/2$ is the area of a regular hexagon of unit side length). The upper bound \eqref{the energy upper bound} on $\psi(N)$ implies the sharpness of \eqref{hales lower bound}, and clarifies its connection with the honeycomb geometry.

\subsection{Low-energy clusters and main result} Motivated by \eqref{the energy upper bound}, we say that a planar, unit-area $N$-cluster $\E$ is a {\bf low-energy cluster with exterior energy density $M$}, if
\begin{equation}
\label{M low}
P(\E)\leq (12)^{1/4}\,N +M\,\sqrt{N}\,.
\end{equation}
This condition amounts in asking that that the ``bulk'' of the cluster $\E$, that is, the set $E=\R^2\setminus\E(0)$, has perimeter $P(E)={\rm O}(\sqrt{|E|})$, while the ``internal perimeter'' of $\E$, that is, the length of $\pa\E$ due to interfaces $\pa\E(h)\cap\pa\E(k)$ with $h\ne k$, $h,k\ne 0$, is approximately that of an $\sqrt{N}\times\sqrt{N}$-chunk of the ideal honeycomb $\H$.

\medskip

By \eqref{the energy upper bound}, unit-area isoperimetric clusters satisfy \eqref{M low}. At the same time, \eqref{M low} rules out many $N$-clusters that are only {\it locally}, but not globally, isoperimetric, and that may fail to look like honeycombs; see
\begin{figure}
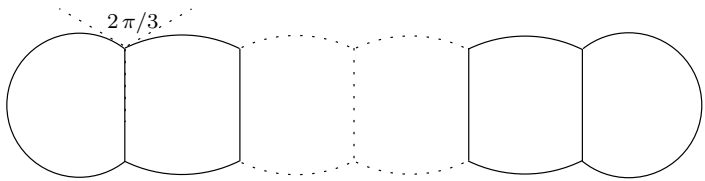\caption{{\small The idea behind the low-energy condition \eqref{M low} is that it identifies unit-area $N$-clusters whose ``internal perimeter'' is comparable to that of an $\sqrt{N}\times\sqrt{N}$-chunk of ideal honeycomb, and whose ``external perimeter'' is comparable to $\sqrt{N}$ (i.e., the square root of the area of the bulk of the cluster). Unit-area locally minimizing clusters may fail to satisfy this condition. For example, the $N$-cluster depicted here satisfies, for some $r_0>0$, the local isoperimetric property $P(\E)\le P(\F)$ for every $\F$ with $|\F|=|\E|$ with $\diam(\F(h)\Delta\E(h))\le r_0$, $h=1,...,N$; but it does not satisfy \eqref{M low} -- as it can be seen, for example, by looking at its external perimeter, which is ${\rm O}(N)$, compare with \eqref{external perimeter} in Theorem \ref{thm main}.}}\label{fig figure}
\end{figure}
Figure \ref{fig figure}.

\medskip

In summary, \eqref{M low} identifies a large class of $N$-clusters which, although not possessing any isoperimetric property, may still be expected, for $N$ large, to be close to honeycombs. Our main result confirms this expectation in a class $\C(N,M)$ of planar unit-area $N$-clusters that {\it contains every unit-area isoperimetric $N$-cluster with simply connected chambers} (i.e, {\it every} unit-area isoperimetric cluster, should conjecture \eqref{conjecture simply connected} hold true); see Remark \ref{remark CNM} below.

\medskip

We say that a unit-area $N$-cluster $\E$ in $\R^2$ belongs to $\C(N,M)$ if:
\begin{enumerate}
\item[(C1)] each $\E(h)$ has finitely many connected components $\{E_j^h\}_{j=1}^{N_h}$, called {\bf cells} (of $\E(h)$ or of $\E$, depending on the context); each cell of $\E$ is simply connected;

\item[(C2)] the boundary of each cell $E_j^h$ consists of finitely many bounded Lipschitz injective curves, called {\bf edges} (of $E_j^h$ or of $\E$, depending on the context); endpoints of edges are called {\bf vertexes};

\item[(C3)] each vertex is the endpoint of {\it exactly three} edges, and each cell has {\it at least two} edges;

\item[(C4)] the boundary $\pa\E$ of $\E$ is connected and satisfies \eqref{M low}, that is
\[
P(\E)\leq (12)^{1/4}\,N +M\,\sqrt{N}\,;
\]

\item[(C5)] denoting by $k_j^h=k_j^h(\E)$ the number of edges of the cell $E_j^h$ of $\E$, we require that
\begin{equation}
\label{bound on cells area}
|E_j^h| \ge \frac1{100}\,,\qquad\mbox{if $2\le k_j^h\le 6$}\,.
\end{equation}
\end{enumerate}

\begin{remark}[The class $\C(N,M)$ and isoperimetric clusters]\label{remark CNM}
  {\rm If $\E$ is a unit-area $N$-isoperimetric cluster, then: (i) each chamber of $\E$ has finitely many connected components, so that the first part of (C1) holds true; (ii) each cell of $\E$ is bounded by finitely many circular arcs/segments, therefore (C2) holds true; (iii) the validity of (C3) is a simple consequence of the Plateau laws; (iv) a sliding argument shows that $\pa\E$ is connected, so that $\E$ satisfies (C4) thanks to \eqref{the energy upper bound}. As a consequence, if each chamber $\E(h)$ of $\E$ has just one single (simply connected) component $E_1^h$, then (C1) and (C5) hold true (with $|E_1^h|=1$), and $\E\in\C(N,M)$ for some $M$. In summary, {\it every unit-area $N$-isoperimetric cluster with simply connected chambers} is an element of $\C(N,M)$ for some $M$.}
\end{remark}

As a last step towards the statement of our main result we introduce the following notation. For $\E\in\C(N,M)$ we introduce the family ${\rm Hex}(\E)$ of those (indexes of) chambers of $\E$ that {\bf are connected and have six edges},
\begin{eqnarray}\label{def of Hex}
{\rm Hex}(\E):=\Big\{h:\, N_h=1\,,k_1^h=6\Big\}\,,
\end{eqnarray}
and, for each $k\ge 2$, the family ${\rm Ch}_k(\E)$ of those (indexes of) chambers of $\E$ {\bf having at least one cell with $k$-sides}
\begin{eqnarray}
\label{def of Chk}
{\rm Ch}_k(\E):=\Big\{h:\mbox{$\exists\,j\in\{1,...,N_h\}$ s.t. $k_j^h=k$}\Big\}\,.
\end{eqnarray}
We also define the {\bf exterior perimeter of $\E$} by setting
\begin{equation}
  \label{def of Pext}
  P_{{\rm ext}}(\E):=P(E_1^0)\,,
\end{equation}
where, we recall, $E_1^0$ is the unique connected component of the exterior chamber $\E(0)$ with infinite area; the {\bf number of exterior edges of $\E$} is defined by setting
\begin{eqnarray}\label{def of edge ext}
{\rm edge}_{\rm ext}(\E):=k_1^0\,;
\end{eqnarray}
and the {\bf interior void of $\E$} as
\begin{equation}
  \label{def of void}
  \E_{\rm void}:=\E(0)\setminus E_1^0\,.
\end{equation}
(Notice that this set must have finite area). Finally, we denote by
\begin{equation}
  \label{dist hex}
  \d_{{\rm hex}}(E)=\inf\Big\{|E\Delta (x+r e^{i\,\theta}[H])|:x\in\R^2,\ \theta\in[0,2\,\pi], \ r^2=|E|\Big\}\,,
\end{equation}
the distance in area of $E\subset\R^2$ from a regular hexagon with area $|E|$ (here $H$ denotes a reference unit-area regular hexagon and $e^{i\theta}$ the angle $\theta$ counter-clockwise rotation of $\R^2$ around the origin). With this terminology in place, we state our main theorem.

\begin{theorem}[Honeycomb-like structure of low-energy clusters]\label{thm main}
There exists a computable constant $C_0$ with the following property. If $N\ge 2$ and $\E\in\C(N,M)$, then
\begin{eqnarray}
  \label{new conclusion 1}
  \#\,{\rm Hex}(\E)\!\!&\ge&\!\! N-C_0\,M\,\sqrt{N}\,,
\end{eqnarray}
with
\begin{eqnarray}
  \label{new conclusion 2}
  \frac1{\#\,{\rm Hex}(\E)}\,\sum_{h\in {\rm Hex}(\E)}\,\d_{{\rm hex}}\big(\E(h)\big)^2\!\!&\le&\!\! \frac{C_0\,M}{\sqrt{N}}\,,
\end{eqnarray}
and
\begin{eqnarray}\label{external perimeter}
P_{\rm{ ext}}(\E)\!\!&\le&\!\! C_0\,M\,\sqrt{N}\,,
\\
\label{external edges}
{\rm edge}_{{\rm ext}}(\E)\!\!&\le&\!\!C_0\,M\,\sqrt{N}\,,
\\
\label{interior void}
|\E_{\rm void}|\!\!&\le&\!\!C_0\,M\,\sqrt{N}\,,
\\
\label{non hexagonal components}
\#\,{\rm Ch}_k(\E)\!\!&\le&\!\!\frac{C_0\,M}{|k-6|}\,\sqrt{N}\,.
\end{eqnarray}
Finally, there is at least one $k\le 5$ such that ${\rm Ch}_k(\E)\ne\varnothing$.
\end{theorem}

\begin{remark}
{\rm Conclusions \eqref{new conclusion 1} and \eqref{new conclusion 2} state that the the vast majority of the chambers of $\E$ are (simply) connected, posses six curvilinear edges, and are close to be regular hexagons. Conclusion \eqref{external perimeter} implies that only a very small fraction of the total cluster perimeter $P(\E)={\rm O}(N)$ is used to compound the bulk of the chambers of $\E$. Given conclusion \eqref{external perimeter}, conclusion \eqref{external edges} indicates that the generic external edge of $\E$ must have length of order one. Simple examples show that {\it all these conclusions are sharp} in the class $\C(N,M)$.}
\end{remark}

\begin{remark}[Connection with the connectedness conjecture]
  {\rm Thanks to Remark \ref{remark CNM}, should one be able to prove the validity of the area lower bound \eqref{bound on cells area} for the cells of isoperimetric clusters, then conclusion \eqref{new conclusion 1} in Theorem \ref{thm main} would imply, in particular, a partial answer to the connectedness conjecture \eqref{conjecture simply connected}, namely, that for a unit-area isoperimetric $N$-cluster, $(N-{\rm O}(\sqrt{N}))$-many chambers are connected.}
\end{remark}

\subsection{Hales' hexagonal isoperimetric inequality and strategy of proof} The proof of Theorem \ref{thm main} is based on a careful extension of the methods developed by Hales in \cite{hales}. The key result in Hales' paper is an inequality for immersed planar curves, called here {\it Hales' hexagonal isoperimetric inequality}, see \eqref{hexagon II hales improv} below. The key step in proving Theorem \ref{thm main} is obtaining a quantitative improvement of Hales' hexagonal isoperimetric inequality. Thus, in order to illustrate our strategy of proof, we need to introduce Hales' hexagonal isoperimetric inequality.

\medskip

Hales' hexagonal isoperimetric inequality is a direct improvement of the hexagonal isoperimetric inequality. Denoting by $\PP_k$ the family of planar polygons with $k$-sides ($k$-polygons) and by $p(k)$ the perimeter of a reference unit-area regular $k$-polygon, the {\bf $k$-polygonal isoperimetric inequality} states that
\begin{equation}
  \label{isoperimetric inequality kgons}
  P(\Pi)\ge p(k)\,\sqrt{|\Pi|}\,,\qquad \forall\Pi\in\PP_k\,,
\end{equation}
with equality if and only if $\Pi$ is a regular $k$-polygon. Here $p(k)$ is explicitly given by
\begin{equation}
    \label{p kappa}
    p(k)=2\,\sqrt{k\,\tan(\pi/k)}\,.
\end{equation}
The case $k=6$ of \eqref{isoperimetric inequality kgons} is of course the {\bf hexagonal isoperimetric inequality},
\begin{equation}
  \label{hexagon II}
  P(\Pi)\ge 2\,(12)^{1/4}\,\sqrt{|\Pi|}\,,\qquad\forall\Pi\in\PP_6\,,
\end{equation}
with equality if and only if $\Pi$ is a regular hexagon. Now, on noticing that $t\mapsto 2\,\sqrt{t\,\tan(\pi/t)}$ is strictly decreasing and convex, and setting
\[
a(k)=\frac{p(k)-p(6)}{6-k}\,,\qquad\forall\, k\ne 6\,,
\]
(so that $a(k)>0$), we can deduce from \eqref{isoperimetric inequality kgons} that if $a\in(a(7),a(5))$, then
\begin{equation}
  \label{hexagon II toth improv}
  P(\Pi)+a\,(k-6)\ge 2\,(12)^{1/4}\,\sqrt{|\Pi|}\,,\qquad\forall\Pi\in\bigcup_{k=3}^\infty \PP_k\,,
\end{equation}
with equality if and only if $\Pi$ is a regular hexagon. This is of course a rather trivial generalization of \eqref{hexagon II}, but it is already sufficient to prove the lower bound $P(\E)>(12)^{1/4}\,N$ on the special class of {\bf unit-area polygonal Plateau-type $N$-clusters $\E$} whose chambers have unit area and consist of finitely many connected polygonal cells, and whose vertexes are the endpoints of {\it exactly three segments}. Indeed, since all the cells $E_j^h$ of $\E$ with $(h,j)\ne (0,1)$ are polygons with finite area, we can apply \eqref{hexagon II toth improv} on them. Denoting by $k_j^h$ the number of sides of $E_j^h$, we thus find that
\begin{eqnarray}\label{lower bound start}
2\,P(\E)\ge \sum_{(h,j)\ne(0,1)}P(E_j^h)\ge   2\,(12)^{1/4}\,\sum_{(h,j)\ne(0,1)}\sqrt{|E_j^h|}+a\,\sum_{(h,j)\ne(0,1)}(6-k_j^h)\,.
\end{eqnarray}
On the one hand, since $|E_j^h|\le|\E(h)|=1$ for $h=1,...,N$ implies $|E_j^h|\le\sqrt{|E_j^h|}$, we get
\begin{equation}
\label{total area}
\sum_{(h,j)\ne(0,1)}\sqrt{|E_j^h|} = \sum_{j=2}^{N_0}\sqrt{|E_j^0|}+\sum_{(h,j),\ h\neq 0}\sqrt{|E_j^h|}\geq\sum_{j=2}^{N_0} \sqrt{|E_j^0|}+\sum_{h=1}^N|\E(h)|\ge N\,.
\end{equation}
On the other hand, counting that each vertex of $\pa\E$ is the endpoint of three segments, we can associate to $\pa\E$ a regular graph on $\mathbb{S}^2$ with numbers of faces $F$, edges $E$, and vertexes $V$ given by
 \[
 F=\sum_{h=0}^NN_h\,,\qquad E=\frac12\sum_{h=0}^N\sum_{j=0}^{N_h}k_j^h\,,\qquad V=\frac23\,E\,,
 \]
so that, by Euler's formula $2=V-E+F$,
 \[
  2=-\frac16\sum_{h=0}^N\sum_{j=0}^{N_h}k_j^h+\sum_{h=0}^NN_h\,,\quad\mbox{that is}\quad 12=\sum_{h=0}^N\sum_{j=1}^{N_h}(6-k_j^h)\,,
  \]
and thus
\begin{equation}
\label{euler formula 2}
\sum_{(h,j)\ne(0,1)}(6-k_j^h)=6+k_1^0\,.
\end{equation}
The combination of $a>0$, \eqref{lower bound start}, \eqref{total area}, and \eqref{euler formula 2} gives $P(\E)>(12)^{1/4}\,N$ for every unit-area polygonal Plateau-type $N$-cluster $\E$ as defined above.

\medskip

The above argument, due to Fejes T\'oth \cite{fejes43}, has been presented in detail since the class of unit-area isoperimetric $N$-clusters $\E$ considered in Hales' isoperimetric principle \eqref{hales lower bound} is actually not that far from the class of unit-area polygonal Plateau-type $N$-clusters: the only difference is that, in the former case, cells may be bounded by circular arcs (with possible non-zero curvature) rather than just by segments (as in the latter). Hales' hexagonal isoperimetric inequality consists of an extension of \eqref{hexagon II toth improv} that suffices to repeat Fejes T\'oth's argument on isoperimetric clusters.

\medskip

The following notation will be needed in stating Hales' hexagonal isoperimetric inequality. Given $s,t\in\mathbb{S}^1=\{z\in\mathbb{C}:|z|=1\}$, $s\ne t$, we denote by $[s,t]$ the set of points in $\mathbb{S}^1$ obtained by moving from $s$ to $t$ in the orientation of $\mathbb{S}^1$ induced by $e^{i\,\theta}$. We say that $I$ is an {\bf interval of $\mathbb{S}^1$}, if $I=\mathbb{S}^1$ or $I=[s,t]$ for some $s\ne t$. If $\g\in{\rm Lip}(I;\R^2)$, then we denote by $L(\g)$ and $A(\g)$ the {\bf length} and {\bf oriented area} of $\g$, defined by setting
\[
L(\gamma)=\int_I |\gamma'(t)| \, dt\,,\qquad A(\gamma) = \int_{\gamma} x \, dy=\int_I\,\gamma^{(1)}(t)\,(\gamma^{(2)})'(t)\,dt\,.
\]
(Here and in the following $x^{(i)}$ is the $i$-th component of $x\in\R^2$. Also, with this convention, if $\g(\theta)=e^{i\,\theta}$ ($\theta\in\mathbb{S}^1$), then $A(\g)=\pi$, and if $\g(\theta)=e^{-i\,\theta}$ ($\theta\in\mathbb{S}^1$), then $A(\g)=-\pi$.) Given $x,y\in\R^2$ we denote by $\ll x,y\rr$ a constant speed parametrization (over a closed interval in $\mathbb{S}^1$) of the segment $\{t\,x+(1-t)\,y:t\in[0,1]\}$. With this notation, given $\g\in\Lip(\mathbb{S}^1;\R^2)$ and $[s,t]\subset\mathbb{S}^1$, we introduce the {\bf signed secant area of $\g$ over $[s,t]$} by setting
\begin{equation}
  \label{alfa}
  \a(\g,s,t)=A\left(\g|_{[s,t]}\right)+A\left(\ll\g(t),\g(s)\rr\right)\,;
\end{equation}
see
\begin{figure}
\label{fig gamma}
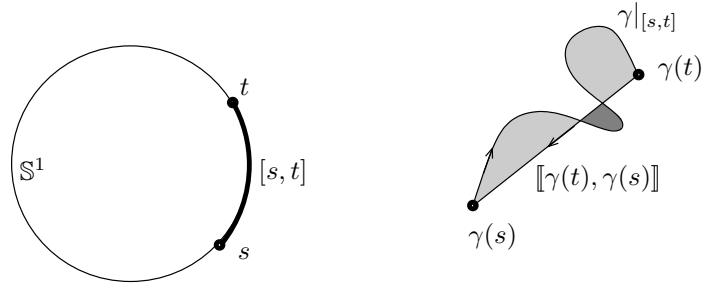
\caption{{\small In the example in the picture, the quantity $\a(\g,s,t)$ defined in \eqref{alfa} is obtained by subtracting the areas depicted in light grey from the areas depicted in dark grey.}}
\end{figure}
Figure \ref{fig gamma}. Finally, given $k\ge 2$, we say that $\tt:\{1,...,k\}\to\mathbb{S}^1$ defines an {\bf ordered $k$-partition of $\mathbb{S}^1$} if, for every $i=1,...,k-1$, $[\tt(i),\tt(i+1)]$ does not contain any $\tt(j)$ with $j\ne i,i+1$. Denoting by $\P_k$ the set of ordered $k$-partitions of $\mathbb{S}^1$, and setting $\tt(k+1)=\tt(1)$ if $\tt\in\P_k$, we define a functional
\[
\s:\Lip(\mathbb{S}^1;\R^2)\times\P_k\to\R\,,
\]
by taking
\begin{equation}
  \label{def of sigma gtt}
  \s(\g,\tt)
=\sum_{i=1}^{k} \min\Big\{\frac12,\max\Big\{-\frac12,\a\Big(\g,\tt(i),\tt(i+1)\Big)\Big\}\Big\}\,.
\end{equation}
Thus $\s(\g,\tt)$ is obtained by adding up the signed secant areas defined by $\g$ and the nodes of the ordered $k$-partition $\tt$, truncated in the interval $[-1/2,1/2]$.
\begin{theoremletter}[Hales' hexagonal isoperimetric inequality, \cite{hales}]\label{thm hales HI} {\it There exists a constant $a>0$ with the following property.  If $k\geq 2$, $(\gamma,\tt)\in\Lip(\mathbb{S}^1;\R^2)\times \mathcal{P}_k$, and
\begin{equation}
\label{Agamma hales bound}
   A(\g)\geq \frac{2\pi}{\sqrt{3}\,k^2}\,,
\end{equation}
then
\begin{equation}
\label{hexagon II hales improv}
L(\gamma) + a\,(k-6) + (12)^{1/4}\, \sigma(\gamma, {\bf{t}}) \geq 2(12)^{1/4}\,\min\{1,A(\gamma)\}
\end{equation}
with equality if and only if $(\gamma,\tt)$ corresponds\footnote{That is, $\gamma$ is an injective parametrization of the boundary of a unit-area regular hexagon and $\{\gamma(\tt(i))\}_{i=1}^6$ are the vertexes of said hexagon.} to a unit-area regular hexagon.}
\end{theoremletter}

\begin{remark}\label{remarks about hales improv} {\rm {\bf (i):} Hales proves \eqref{hexagon II hales improv} with $a=0.0505$. Notice that this choice satisfies $a\in(a(7),a(5))$, and is thus admissible in \eqref{hexagon II toth improv}. In particular, when $\gamma$ parameterizes the boundary of a {\it unit-area} polygon and $\tt$ ranges through the vertexes of that polygon, then \eqref{hexagon II hales improv} boils down to \eqref{hexagon II toth improv}; {\bf (ii):} One cannot take a different value than $(12)^{1/4}$ for the constant in front of $\sigma(\g,\tt)$: this can be verified by taking $k=6$, $\g$ a parametrization of a unit-area regular hexagon $H$, $\tt$ ranging through the vertexes of $H$, and then modifying $\g$ by taking small circular variations of the edges of $H$ and subsequent rescalings; {\bf (iii):} The inequality cannot hold with $A(\g)$ in place of $\min\{1,A(\gamma)\}$ because of the different scaling of $L(\g)$ and $A(\g)$ -- of course, a variant of \eqref{hexagon II hales improv} with $\sqrt{A(\g)}$ in place of $\min\{1,A(\gamma)\}$ would have looked more natural; {\bf (iv):} taking $k<6$, and testing \eqref{hexagon II hales improv} with $t\,\g$ in place of $\g$, we eventually obtain a contradiction as $t\to 0^+$; this is why a lower bound on $A(\g)$ must be assumed.}
\end{remark}

\begin{remark}[Proof of \eqref{hales lower bound} and assumption \eqref{Agamma hales bound}]
  \label{remark about the lb} {\rm It seems useful to recall how Theorem \ref{thm hales HI} is used in proving \eqref{hales lower bound}, as this point allows us to illustrate the reason for the particular form \eqref{Agamma hales bound} of the lower bound on $A(\gamma)$ assumed in Theorem \ref{thm hales HI}. To this end, let us introduce the functional
  \[
  \de(\F)=P(\F)-(12)^{1/4}\,\sum_{h=1}^N\min\left\{1,|\F(h)|\right\}\,,\qquad\mbox{$\F$ an $N$-cluster}\,.
  \]
  When $\E$ is a unit-area isoperimetric $N$-cluster, then $\de(\E)=P(\E)-(12)^{1/4}\,N$, and \eqref{hales lower bound} is equivalent to $\de(\E)>0$. As Hales argues in \cite[Remark 2.7]{hales}, if a cell $E_j^h$ of $\E$ is such that $|E_j^h|< 2\,\pi/(\sqrt3\,(k_j^h)^2)$, then a new $N$-cluster $\F$ can be defined so that $\de(\F)\le\de(\E)$ and $\F(h)$ has one less cell than $\E(h)$ (one simply has to add $E_j^h$ to a cell $E_i^m$, $m\ne h$, sharing with $E_j^h$ its longest edge). In finitely many steps, one constructs an $N$-cluster $\G$ such that $\de(\G)\le\de(\E)$, each cell of $\G$ is simply connected and satisfies $|G_j^h|\ge 2\,\pi/(\sqrt3\,(k_j^h)^2)$. Theorem \ref{thm hales HI} can then be applied to each cell of $\G$, and, adding up over the number of cells, Fejes T\'oth's argument can be repeated with minor modifications to conclude that $\de(\G)$, thus $\de(\E)$, is positive (compare with the proof of Theorem \ref{thm main}).}
\end{remark}

Coming back to the presentation of the proof of our main result, Theorem \ref{thm main}, the key ingredient will be obtaining the following {\it quantitative improvement} of Hales' hexagonal isoperimetric inequality \eqref{hexagon II hales improv}.

\begin{theorem}[A quantitative Hales' hexagonal isoperimetric inequality]\label{thm quantita hales HI} There exist positive constants $a_1$, $a_2$, and $a_3$ such that, setting $a_3(k)=0$ if $k\ne 6$ and $a_3(6)=a_3$, then the following holds. If $k\ge 2$, $(\g,\tt)\in\Lip(\mathbb{S}^1;\R^2)\times \P_{k}$, $\g$ is injective, $A(\g)\le 1$, and
\begin{equation}
\label{Agamma our bound}
A(\g)\ge\frac1{100}\,,\qquad\mbox{if $2\le k\le 6$}\,,
\end{equation}
then
\begin{equation}\label{hexagon II our improv}
  \begin{split}
  L(\g)+a_1\,(k-6)+(12)^{1/4}\,\s(\g,\tt)&\ge2\,(12)^{1/4}\,A(\g)
  \\
  &\,\,\,+a_2\,|k-6|
  \\
  &\,\,\,+a_3(k)\,\Big\{\d_{{\rm hex}}(E_{\gamma})^2+\big(1-A(\g)\big) \Big\}\,,
  \end{split}
\end{equation}
where $E_{\gamma}$ denotes the bounded connected component of $\R^2\setminus \gamma(\mathbb{S}^1)$ defined by $\g$ according to Jordan's theorem.
\end{theorem}

\begin{remark}\label{remarks about our improv}
  {\rm {\bf (i):} With respect to assumption \eqref{Agamma hales bound} of Theorem \ref{thm hales HI}, we assume no lower bound on $A(\g)$ when $k\ge 7$, and we weaken the lower bound on $A(\g)\ge 2\,\pi/(\sqrt{3}\,k^2)$ when $2\le k\le 6$ (indeed $2\,\pi/(\sqrt{3}\,(6^2))>1/10$); {\bf (ii):} A more cosmetic than substantial change is the replacement of the term $2\,(12)^{1/4}\,\min\{1,A(\g)\}$ in \eqref{hexagon II hales improv} with the term $2\,(12)^{1/4}\,A(\g)$ in \eqref{hexagon II our improv}. By Remark \ref{remarks about hales improv}-(iii), this change calls for adding the assumption $A(\g)\le 1$; {\bf (iii):} The second and third term on the right-hand side of \eqref{hexagon II our improv} quantify the distance of $(\g,\tt)$ from corresponding to a unit-area regular hexagon (unique equality cases of \eqref{hexagon II hales improv}).}
\end{remark}

\subsection{Organization of the paper} In Section \ref{section proof main}, Theorem \eqref{thm main} is deduced from Theorem \ref{thm quantita hales HI}. After a brief introduction of the $\arc$ function in Section \ref{section arc}, the rest of the paper is thus concerned with the proof of Theorem \ref{thm quantita hales HI}, which is broken down in two sections. In Section \ref{section proof of improved hales} we present the part of the argument making use of quantitative isoperimetry methods, while in Section \ref{section hales intermed improv} we present the part of the argument which follows more closely \cite{hales}.

\medskip

\noindent {\bf Acknowledgement:} MC has been partially supported by the fund ``Analisi di problemi variazionali e differenziali, teoria degli operatori (F.S.R Politecnico di Milano)". MC thanks the financial support of PRIN 2022R537CS ``Nodal optimization, nonlinear elliptic equations, nonlocal geometric problems, with a focus on regularity” funded by the European Union under Next Generation EU. KDM is supported by NSF-DMS RTG 1840314 and the NSF Graduate Research Fellowship Program under Grant DGE 2137420. FM is supported by NSF-DMS RTG 1840314 and NSF-DMS 2247544. This work was partially completed while KDM was hosted at the math department of Politecnico di Milano.

\section{From the quantitative Hales' inequality to the main theorem}\label{section proof main} In this section we show how to deduce Theorem \ref{thm main} from Theorem \ref{thm quantita hales HI}. We shall use ``Dido's inequality": {\it if a curve of length $\ell$ bounds an area $a$ with its chord, then $\ell\ge\sqrt{2\pi\,a}$}. This is easily proved by reflecting the circular arc with respect to the chord and by applying the isoperimetric inequality to find $2\,\ell\ge2\sqrt{\pi}\sqrt{2a}$.

\begin{proof}
  [Proof of Theorem \ref{thm main}] Let $\E\in\C(N,M)$, let $\{E_j^h\}_{j=1}^{N_h}$ be the cells of the chamber $\E(h)$, and let $k_j^h$ be the number of vertexes of $E_j^h$ (see the defining properties (C1)--(C5) of $\C(N,M)$ for the notation and terminology used here). For each cell $E_j^h$ we can find an injective $\g_j^h\in \Lip(\mathbb{S}^1;\R^2)$ and $\tt_j^h\in\P_{k_j^h}$ such that:

  \medskip

  \noindent (i): $\g_j^h(\tt_j^h(i))$ is the $i$-th vertex of $E_j^h$;

  \medskip

  \noindent (ii): if $(h,j)\ne (0,1)$ (and thus $E_j^h$ is not the cell $E_1^0$ of $\E$ with infinite area), then the orientation of $\g_j^h$ can be chosen so that $A(\g_j^h)=|E_j^h|$ and $L(\g_j^h)=P(E_j^h)$;

  \medskip

  \noindent (iii): the orientation of $\g_1^0$ is such that
  \begin{equation}
  \label{the sum}
  \sum_{(h,j)\ne(0,1)}\s(\g_j^h,\tt_j^h)=-\s(\g_1^0,\tt_1^0)\,,
  \end{equation}
  where we are using the cancelations due to opposite orientations along each internal edge of $\E$.

  \medskip

  By (ii), since $P(E_j^h)=L(\g_j^h)$, we notice that
  \begin{equation}
    \label{a0}
      2\,P(\E)=P(E_1^0)+\sum_{(h,j)\ne(0,1)}L(\g_j^h)\,,
  \end{equation}
  and bound from below the two terms on the right-hand side as follows. Concerning $P(E_1^0)$, setting for brevity
  \[
  A_i=\a\left(\gamma_1^0,\tt_1^0(i),\tt_1^0(i+1)\right)\,,
  \]
  (see \eqref{alfa}) and applying Dido's inequality, we find that
  \begin{equation}
    \label{PE10 first}
      \frac{P(E_1^0)}{\sqrt{2\,\pi}}\ge\sum_{i=1}^{k_1^0}\sqrt{|A_i|}
      \ge\sum_{i=1}^{k_1^0}\sqrt{\min\left\{1/2,|A_i|\right\}}
      \ge\sum_{i=1}^{k_1^0}\min\left\{1/2,|A_i|\right\}  \ge|\s(\g_1^0,\t_1^0)|\,.
  \end{equation}
  Assumption \eqref{Agamma our bound} of Theorem \ref{thm quantita hales HI} holds for each $\g_j^h$ with $(h,j)\ne (0,1)$ thanks to property (C5).  We can thus apply \eqref{hexagon II our improv} to each $\g_j^h$ and adding up the results we find
  \begin{eqnarray}\label{a1}
  &&\sum_{(h,j)\ne(0,1)}L(\g_j^h)\ge
  \sum_{(h,j)\ne(0,1)}\left\{
  a_1\,(6-k_j^h)-(12)^{1/4}\,\s(\g_j^h,\tt_j^h)\right\}
  \\\nonumber
  &&+\sum_{(h,j)\ne(0,1)}\,\Big\{2(12)^{1/4}\,|E_j^h|+a_2\,|k_j^h-6|
  +a_3(k_j^h)\,\Big(\d_{{\rm hex}}(E_j^h)^2+\big(1-|E_j^h|\big) \Big)\Big\}\,.
  \end{eqnarray}
  Property (C3) of $\C(N,M)$ ensures that we can repeat Fejes T\'oth's argument to deduce \eqref{euler formula 2}, that is $\sum_{(h,j)\ne(0,1)}
  (6-k_j^h)=6+k_1^0$. Combining this last identity,
  \[
  \sum_{(h,j)\ne(0,1)}|E_j^h|=N+|\E_{\rm void}|
  \]
  (recall \eqref{def of void}), and \eqref{the sum} into \eqref{a1}, and then recalling \eqref{a0}, we find that
  \begin{eqnarray}\nonumber
    2\,P(\E)&\ge&2(12)^{1/4}\,N+2\,(12)^{1/4}\,|\E_{\rm void}|+P(E_1^0)+(12)^{1/4}\,\s(\g_1^0,\t_1^0)+a_1\,(6+k_1^0)
    \\
    \label{a3}
    &&+\sum_{(h,j)\ne(0,1)}\,\Big\{a_2\,|k_j^h-6|+a_3(k_j^h)\,\Big(\d_{{\rm hex}}(E_j^h)^2+\big(1-|E_j^h|\big) \Big)\Big\}\,.
  \end{eqnarray}
  Now, taking into account that $\sqrt{2\pi}-(12)^{1/4}>0$, thanks to \eqref{PE10 first} we find that
  \begin{equation}
    \label{PE10 second}
      P(E_1^0)+(12)^{1/4}\,\s(\g_1^0,\t_1^0)\ge 2\,c_0\,P(E_1^0)\,,\qquad\mbox{if}\qquad c_0=\frac{\sqrt{2\pi}-(12)^{1/4}}{2\,\sqrt{2\pi}}\,.
  \end{equation}
  By combining \eqref{a3} and \eqref{PE10 second} with the low-energy condition $P(\E)\le(12)^{1/4}\,N+M\,\sqrt{N}$, we immediately deduce
  \begin{align}
  &2\,(12)^{1/4}\,|\E_{\rm void}|+a_1\,(6+k_1^0)+c_0\,P(E_1^0) +a_2\, \sum_{(h,j)\ne(0,1)}|k_j^h-6|  \, \le  2M\,\sqrt{N}\,,
  \label{pf1}
  \\
  & \sum_{(h,j)\ne(0,1)} a_3(k_j^h)\,\Big(\d_{{\rm hex}}(E_j^h)^2 +\big(1-|E_j^h|\big)\Big)\, \le  2M\,\sqrt{N}\,.
  \label{pf2}
  \end{align}
  Conclusions \eqref{external perimeter}, \eqref{external edges}, \eqref{interior void}, and \eqref{non hexagonal components} follow immediately from \eqref{pf1} and
  \[
  |k-6|\,\#({\rm Ch}_k(\E))\le \sum_{(h,j)\ne(0,1)}|k_j^h-6|\,.
  \]
  To prove \eqref{new conclusion 1} and \eqref{new conclusion 2}, we begin by noticing that ${\rm Hex}(\E)=I_1\cap I_2$ where
  \begin{eqnarray*}
  &&I_1=\Big\{h: k_j^h=6,\,\forall j\in\{1,...,N_h\}\Big\}\,,
  \\
  &&I_2=\Big\{h:N_h=1\Big\}\,.
  \end{eqnarray*}
  Now, by \eqref{pf1}, setting $I=\{1,...,N\}$, we find
  \[
  \#\,(I\setminus I_1)\le \sum_{\big\{(h,j):k_j^h\ne 6\big\}}|k_j^h-6|\le  \sum_{(h,j)\ne(0,1)}|k_j^h-6| \le C\,M\,\sqrt{N}\,,
  \]
  and, similarly, by \eqref{pf2}, we obtain
  \begin{eqnarray*}
  C\,M\,\sqrt{N}&\ge&\sum_{\big\{(h,j)\ne (0,1):k_j^h=6\big\}}\big(1-|E_j^h|\big)\ge\sum_{h\in I_1}\sum_{j=1}^{N_h}\,\big(1-|E_j^h|\big)
  \\
  &=&\sum_{h\in I_1}\big(N_h-1\big)\ge\#(I_1\setminus I_2)\,.
  \end{eqnarray*}
  We thus conclude that
  \[
  \#{\rm Hex}(\E)=\#(I_1)-\#(I_1\setminus I_2)=\#I-\#\,(I\setminus I_1)-\#(I_1\setminus I_2)\ge N-C\,M\,\sqrt{N}\,,
  \]
  that is \eqref{new conclusion 1}. Of course, \eqref{new conclusion 2} follows immediately from \eqref{new conclusion 1} and \eqref{pf2}.

  \medskip

  Finally, let us assume, by way of contradiction that $k_j^h\ge 6$ for all $(h,j)\ne (0,1)$. In this case, $k_j^h-6$ being non-negative for every $(h,j)\ne (0,1)$, we can go back to \eqref{a1} and apply a version of \eqref{hexagon II our improv} where, in place of $a_1$, an arbitrarily large constant $L$ appears. Correspondingly, in place of \eqref{pf1} we now deduce an inequality that implies, in particular, $6\,L\le M\,\sqrt{N}$. By taking $L$ large enough in terms of $M$ and $N$ we obtain a contradiction.
\end{proof}

\section{The arc function}\label{section arc} Starting from the next section we will make repeated use of the function
\[
\arc(\ell,x)\,,\qquad \ell\ge0\,, x\ge 0\,,
\]
defined as the length of a circular arc subtending a segment of length $\ell$ and bounding a region of area $x$. Clearly $\arc(0,x)=2\,\sqrt\pi\,\sqrt{x}$ is the isoperimetric profile of $\R^2$. By scaling,
\begin{equation}
  \label{arc scaling}
  \arc(\ell,x)=\ell\,\arc\left(1,\frac{x}{\ell^2}\right)\,,\qquad\forall \ell>0\,, x\ge 0\,,
\end{equation}
so that we can directly focus on $\arc_1=\arc(1,\cdot)$. We claim that
\begin{eqnarray}
\label{implicit1}
  &&\mbox{$\arc_1'>0$ on $(0,\infty)$ with $\arc_1(0)=1$, $\arc_1(+\infty)=+\infty$}\,,
  \\
\label{implicit32}
  &&\arc_1'(0)=0\,,\,\, \arc_1''(0)=12\,,
  \\
\label{implicit2}
  &&\mbox{$\arc_1''>0$ on $[0,\pi/8)$ and $\arc_1''<0$ on $(\pi/8,\infty)$}\,.
\end{eqnarray}
We can obtain implicit formulas for $\arc_1$ that can be used in proving \eqref{implicit1}, \eqref{implicit32}, and \eqref{implicit2}. For example, by combining \eqref{arc scaling} with the identity
\begin{equation}
  \label{implicit}
  2\,\theta\,R=\arc\Big(2\,R\,\sin\theta,R^2\left(\theta-\sin\theta\,\cos\theta\right)\Big)\,,\qquad \theta\in[0,\pi/2]\,,R\ge0\,,
\end{equation}
see
\begin{figure}
  \label{fig arc}
  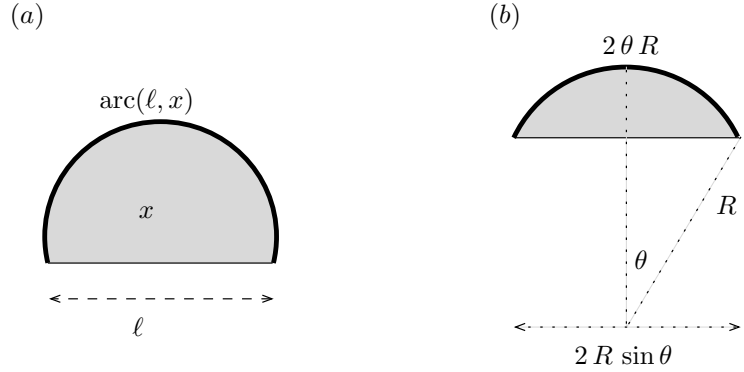
\caption{\small{(a) $\arc(\ell,x)$ is defined as the length of a circular arc (depicted in bold) subtending a chord of length $\ell$ and including a secant area $x$; (b) An implicit formula for $\arc_1$ on the interval $[0,\pi/2]$ can be obtained by referring to this picture. The second argument of $\arc$ in \eqref{implicit} is obtained by subtracting the area of a rectangle with sidelengths $R\,\sin\theta$ and $R\,\cos\theta$ from the area $\theta\,R^2$ of an angular sector whose amplitude equals $2\,\theta$ radians.}}
\end{figure}
Figure \ref{fig arc}, we obtain the following implicit representation\footnote{Notice that $p$ is strictly increasing from $[0,\pi/2]$ to $[0,\pi/8]$.} of $\arc_1$ on the interval $[0,\pi/8]$:
\begin{equation}\label{implicit formula}
\frac{\theta}{\sin\theta}=\arc_1(p(\theta))\,,\qquad\mbox{for}\quad p(\theta)=\frac{\theta-\sin\theta\,\cos\theta}{4\,\sin^2\theta}\,,\,\,\theta\in[0,\pi/2]\,.
\end{equation}
(A similar construction can be used to represent $\arc_1$ on $[\pi/8,\infty)$.) By combining \eqref{implicit formula} with a Taylor expansion we easily prove \eqref{implicit32}. Similarly, we can use \eqref{implicit formula} to prove \eqref{implicit1} and \eqref{implicit2} on $[0,\pi/8]$, and a representation of $\arc_1$ on $[\pi/8,\infty)$ can be used to complete the proof of \eqref{implicit1} and \eqref{implicit2}. Since this approach (although formally correct) is a bit obscure, we prove \eqref{implicit1} and \eqref{implicit2} with the following argument, which seems more transparent. The key remark is that by the variational definition of (mean) curvature (see, e.g., \cite[Remark 17.6]{maggiBOOK}), we have
\begin{equation}
  \label{geometric meaning arc1}
  \begin{split}
    \frac{\pa\arc}{\pa x}(\ell,x)=\,\,&\mbox{curvature of a circular arc}
    \\
    &\mbox{enclosing an area $x$ above a segment of length $\ell$}\,.
  \end{split}
\end{equation}
By \eqref{geometric meaning arc1} (with $\ell=1$) we see that $\arc_1'>0$ on $(0,\infty)$, thus deducing \eqref{implicit1}. Moreover, as $x$ increases from $0$ to $\pi/8$ (with $x=\pi/8$ corresponding to $\theta=\pi/2$ in Figure \ref{fig arc}-(b)), we see from \eqref{geometric meaning arc1} that $\arc_1'$ strictly increases from $0$ to $2$, thus establishing that $\arc_1''>0$ on $(0,\pi/8)$; similarly, as $x$ increases from $\pi/8$ to the limit value $+\infty$, we see from \eqref{geometric meaning arc1} that $\arc_1'$ strictly decreases from $2$ back to the limit value $0$, so that $\arc_1''<0$ on $(\pi/8,\infty)$.

\section{Proof of the quantitative Hales' hexagonal inequality, part one}\label{section proof of improved hales} In this section we begin the proof of Theorem \ref{thm quantita hales HI}. More precisely, we prove Theorem \ref{thm quantita hales HI} {\it conditionally} to the validity of an improvement of Theorem \ref{thm hales HI} (Hales' hexagonal isoperimetric inequality). This improvement of Theorem \ref{thm hales HI} will be established in Section \ref{section hales intermed improv} as Theorem \ref{thm slight improve}, and is based on a refinement of the arguments employed in \cite{hales}. The argument presented in this section, instead, makes use of a quantitative hexagonal isoperimetric inequality proved in \cite[Corollary 2.2]{carocciamaggi} (as an elaboration on \cite[Corollary 1.3]{indreiPOLIGONI}), as stated in Theorem \ref{lem: d_hex bound} below. In the following, given two closed sets $C$ and $K$ in $\R^2$, we denote by
\[
\hd(C,K):=\max\Big\{\sup_{x\in C}{\rm dist}(x,K)\,, \sup_{x\in K}{\rm dist}(x,C)\Big\}\,,
\]
the Hausdorff distance between $C$ and $K$.

\begin{theorem}\label{lem: d_hex bound}
    There exist positive constants $c$ and $\eta$ such that if  $\Pi\subset \mathbb{R}^2$ is a convex hexagon with $\hd(\Pi,H)\leq \eta$ for some regular hexagon $H$, then
    \[
    P(\Pi)-2\,(12)^{1/4}\,\sqrt{|\Pi|}\geq c\, |\Pi\Delta H^*|^2\,,
    \]
    for a regular hexagon $H^*$ with $|H^*|=|\Pi|$.
\end{theorem}

\begin{proof}[Proof of Theorem \ref{lem: d_hex bound}]
  By \cite[Corollary 2.2]{carocciamaggi} (which is stated on the two-dimensional flat torus, but obviously holds on the plane too) there exist positive constants $c'$ and $\eta$ such that if $\Pi\subset \mathbb{R}^2$ is a convex hexagon with $\hd(\partial \Pi, \partial H)\leq \eta$ for some regular hexagon $H$, then, for a regular hexagon $H^*$ with $|H^*|=|\Pi|$,
  \[
  P(\Pi)-2\,(12)^{1/4}\,\sqrt{|\Pi|}\geq c'\, \hd(\partial \Pi, \partial H^*)^2\,.
  \]
  Since $\hd(\partial \Pi, \partial H)=\hd(\Pi, H)$, we conclude the proof by observing that, for a positive constant $C$ (independent of $\Pi$), it holds $|\Pi\Delta H^*|\le C\,\hd(\partial \Pi, \partial H^*)$.
\end{proof}

\begin{proof}[Proof of Theorem \ref{thm quantita hales HI}, part one] We prove Theorem \ref{thm quantita hales HI} conditionally to the validity of Theorem \ref{thm slight improve} (that will be proved in Section \ref{section hales intermed improv}). Theorem \ref{thm slight improve} asserts the following:

\medskip

{\it There exist positive constants $a_1$ and $a_2$ with the following property. If $k\ge 2$, $(\gamma,\tt)\in  \Lip(\mathbb{S}^1;\R^2)\times \mathcal{P}_k$, and}
\begin{equation}
  \label{k condition x}
  A(\g)>0\quad\mbox{if $k=6$}\,,\qquad A(\gamma)\geq \frac1{100}\quad\mbox{if $2\le k\le 5$}\,,
\end{equation}
{\it then}
\begin{equation}
  \label{hales inequality x}
L(\g)+a_1\,(k-6)+(12)^{1/4}\,\s(\g,\tt)\ge 2\,(12)^{1/4}\,\min\{1,A(\g)\}+a_2\,|k-6|\,,
\end{equation}
{\it with equality if and only if $(\gamma,\tt)$ corresponds to a unit-area regular hexagon.}

\medskip

This statement implies Theorem \ref{thm quantita hales HI} when $k\ne 6$. Therefore, for the remaining part of the proof, we shall assume that $k=6$. We thus want to prove the existence of a positive constant $a_3$  such that if $(\gamma,\tt)\in \Lip(\mathbb{S}^1,\R^2)\times \P_6$, $\g$ is injective, and
\begin{equation}
\label{hales assumption 6}
1\ge A(\gamma)\ge\frac1{100}\,,
\end{equation}
then
\begin{eqnarray}
\label{hales inequality improved hexagons}
L(\g)+ (12)^{1/4}\,\s(\g,\tt)\ge 2\,(12)^{1/4}\,A(\g)+a_3\,\Big\{\d_{\mathrm{hex}}(E_{\gamma})^2+\big(1-A(\gamma)\big)\Big\}\,,
\end{eqnarray}
where $E_\g$ denotes the bounded connected component of $\R^2\setminus\g(\mathbb{S}^1)$ identified by $\g$ (thanks to Jordan's theorem).

\medskip

\noindent {\it Step one}: In this step we set, for $(\g,\tt)\in \Lip(\mathbb{S}^1;\R^2)\times \mathcal{P}_6$,
\begin{equation}\label{deficit}
\de(\gamma,\tt)=L(\gamma)
    +(12)^{1/4}\,\s(\gamma,\tt)- 2(12)^{1/4}\,A(\g)\,,
\end{equation}
and prove that for every $\e>0$ there is $\de>0$ such that, if $\de(\g,\tt)\le\de$, then there is a unit-area, regular hexagon $H_0$ with
\begin{equation}\label{eqn:DistSmall}
    \max\Big\{\hd(E_{\gamma},H_0),\hd(\Pi_{\g,\tt},H_0),|E_\g\Delta \Pi_{\g,\tt}|\Big\}  \leq \e\,,
\end{equation}
where $\Pi_{\g,\tt}$ is the convex envelope of $\{\gamma(\tt(i))\}_{i=1}^6$. Notice, in particular, that if $\e$ is small enough, then \eqref{eqn:DistSmall} implies that $\Pi_{\g,\tt}$ is a convex hexagon.

\medskip

We prove this arguing by contradiction, aiming at using our assumption that, under \eqref{k condition x}, \eqref{hales inequality x} holds, and thus $\de(\gamma,\tt)\ge0$ with equality if and only if $k=6$ and $(\g,\tt)$ corresponds to a unit-area, regular hexagon.

\medskip

We thus consider $\e_*>0$ and a sequence $\{(\gamma_j,\tt_j)\}_{j}$ in $\Lip(\mathbb{S}^1;\R^2)\times \mathcal{P}_6$ with $1\ge A(\g_j)\ge 1/100$, $\de(\g_j,\tt_j)\to 0^+$ as $j\to\infty$, and
\begin{equation}
\label{Eh}
\max\left\{\inf_H\left\{\max\{\hd(E_{\gamma_j},H)\,,\hd(\Pi_{\g_j,\tt_j},H)\}\right\},|E_{\gamma_j}\Delta \Pi_{\gamma_j,\tt_j}|\right\}\ge\e_*\,,\qquad\forall j\,.
\end{equation}
(Here $H$ denotes a generic unit-area, regular hexagon.) Up to a reparametrization we can assume that
\begin{equation}
\label{constant speed}
|\g_j'|=\frac{L(\g_j)}{2\,\pi}\qquad\mbox{on $\mathbb{S}^1$}\,.
\end{equation}
By $|\sigma(\gamma_j,\tt_j)|\le1/2$, $A(\g_j)\le1$, and $\de(\g_j,\tt_j)\to 0$ we find that $\ell=\sup_jL(\g_j)<\infty$. By \eqref{constant speed}, up to extracting subsequences and up to translations, there is $\g\in\Lip(\mathbb{S}^1;\R^2)$ such that $\g_j\to\g$ uniformly on $\mathbb{S}^1$ and $\g_j'\weakstar\g'$ in $L^\infty(\mathbb{S}^1;\R^2)$: in particular,
\begin{equation}
  \label{Agj to Ag}
  A(\g_j)=\int_{\mathbb{S}^1}\gamma_j^{(1)}(\gamma_j')^{(2)}\to \int_{\mathbb{S}^1}\gamma^{(1)}(\gamma')^{(2)}=A(\g)
\end{equation}
as $j\to\infty$ (here we are denoting by $x^{(i)}$ the components of $x\in\R^2$), and from the corresponding bounds on $A(\g_j)$ we find that
\[
1\ge A(\g)\ge \frac1{100}\,.
\]
Now, up to extracting subsequences, there is $\{s_i:1\le i\le 6\}\subset\mathbb{S}^1$ such that, as $j\to\infty$, $\tt_j(i)\to s_i$ for each $i=1,...,6$. In particular, for each $i=1,...6$, and setting $s_7=s_1$, we have $s_i\le s_{i+1}$ in the ordering of $\mathbb{S}^1$. To compute the limit of $\s(\g_j,\tt_j)$ we start noticing that
\begin{equation}
  \label{formula A}
\a\left(\g_j,\tt_j(i),\tt_j(i+1)\right)=\int_{[\tt_j(i),\tt_j(i+1)]}\g_j^{(1)}\,(\g_j')^{(2)}
+A\left(\ll\g_j(\tt_j(i+1)),\g_j(\tt_j(i))\rr\right)\,.
\end{equation}
Starting from \eqref{formula A} and thanks to the fact that $\g_j\to\g$ uniformly on $\mathbb{S}^1$ and $\g_j'\weakstar\g'$ in $L^\infty(\mathbb{S}^1;\R^2)$, we see that if $s_i<s_{i+1}$, then
\[
\lim_{j\to\infty}\a\left(\g_j,\tt_j(i),\tt_j(i+1)\right)=\a(\g,s_i,s_{i+1})\,;
\]
while if $s_i=s_{i+1}$ but there is at least one $\ell\ne i$ such that $s_\ell\ne s_{\ell+1}$, then
\begin{equation}
  \label{previous}
  \lim_{j\to\infty}\a\left(\g_j,\tt_j(i),\tt_j(i+1)\right)=0\,;
\end{equation}
and, finally, if $s_i=s_{i+1}$ for all $i$, then there is $i_0$ such that \eqref{previous} holds for all $i\ne i_0$, and
\[
\lim_{j\to\infty}\a\left(\g_j,\tt_j(i_0),\tt_j(i_0+1)\right)=A(\g)\,.
\]
By using these three facts we conclude that, either $s_i=s_{i+1}$ for all $i$ and
\begin{equation}
  \label{either1}
  \lim_{j\to\infty}\s(\g_j,\tt_j)=\min\Big\{\frac12,A(\g)\Big\}\,,
\end{equation}
or there exists $2\le k\le 6$ and $\tt\in\P_k$ such that
\begin{equation}
  \label{either2}
\lim_{j\to\infty}\s(\g_j,\tt_j)=\s(\g,\tt)\,.
\end{equation}
If we are in the case when \eqref{either1} holds, then we deduce from the isoperimetric inequality, $\de(\g_j,\tt_j)\ge0$, \eqref{Agj to Ag}, and \eqref{either1} that
\begin{eqnarray}\nonumber
2\,\sqrt\pi\,\sqrt{A(\g)}&\le&L(\g)\le\liminf_{j\to\infty}L(\g_j)\le
\lim_{j\to\infty}2\,(12)^{1/4}\,A(\g_j)-(12)^{1/4}\,\s(\gamma_j,\tt_j)
\\\label{oddio}
&=&2\,(12)^{1/4}\,A(\g)-(12)^{1/4}\,\min\Big\{\frac12,A(\g)\Big\}\,.
\end{eqnarray}
Now, the function
\[
F(x)=2\,\sqrt{\pi\,x}-2\,(12)^{1/4}\,x+(12)^{1/4}\,\min\{x,1/2\}\,,
\]
is concave on $[0,1]$, with $F(0)=0$ and $F(1)=2\,\sqrt{\pi}-(12)^{1/4}>0$, so that \eqref{oddio} implies $A(\g)=0$ (a contradiction, as $A(\g)\ge 1/100$). Therefore \eqref{either1} never occurs, and by combining \eqref{Agj to Ag} with \eqref{either2}, $\de(\g_j,\tt_j)\ge0$, and $\de(\g_j,\tt_j)\to 0$ as $j\to\infty$ we conclude that $\de(\g,\tt)=0$ for $\tt$ as in \eqref{either2}. By our assumption, $\de(\g,\tt)=0$ implies that $k=6$, $\g$ is a monotone parametrization of the boundary of a regular, unit-area hexagon, $\{\g(\tt(i))\}_{i=1}^6$ are the vertexes of such hexagon, and $\tt_j(i)\to\tt(i)$ for $1\le i\le 6$ as $j\to\infty$. As $\g_j\to\g$ uniformly on $\mathbb{S}^1$ we conclude that
\[
\max\Big\{\inf_H\left\{\max\{\hd(E_{\gamma_j},H)\,,\hd(\Pi_{\g_j,\tt_j},H)\}\right\},|E_{\gamma_j}\Delta \Pi_{\gamma_j,\tt_j}|\Big\}\to 0\,,
\]
as $j\to\infty$, in contradiction with \eqref{Eh}.

\medskip

\noindent  {\it Step two}: We conclude the proof of Theorem \ref{thm quantita hales HI}. Let $(\gamma,\tt)\in \Lip(\mathbb{S}^1,\R^2)\times \P_6$ be such that $\g$ is injective and $1\ge A(\g)\ge 1/100$, and let $E_\g$ be the connected component of $\R^2\setminus\g(\mathbb{S}^1)$ identified by Jordan's theorem (so that $|E_\g|=A(\g)$). We want to prove the existence of $a_3>0$ such that
\begin{eqnarray}
\label{hales inequality improved hexagons proof}
\de(\g,\tt)\ge a_3\,\Big\{\d_{\mathrm{hex}}(E_{\gamma})^2+\big(1-A(\gamma)\big)\Big\}\,.
\end{eqnarray}
Since $\d_{\mathrm{hex}}(E_{\gamma})<2\,|E_\g|\le 2$, if $\de(\g,\tt)\ge\de$ for some constant $\de>0$, then \eqref{hales inequality improved hexagons proof} holds for every $a_3\le\de/5$. For $\e>0$ to be chosen later, we select $\de>0$ depending on $\e$ as determined in step one, and reduce to the situation when there is a unit-area, regular hexagon $H_0$ such that
\begin{equation}
  \label{first red}
  \max\Big\{\hd(E_\g,H_0),\hd(\Pi_{\g,\tt},H_0),|E_\g\Delta\Pi_{\g,\tt}|\Big\}<\e\,,\qquad \mbox{$\Pi_{\g,\tt}$ is a convex hexagon}\,.
\end{equation}
If we denote by $x_i$ the total area enclosed between the $i$-th (curvilinear) edge of $E_\g$ and the corresponding $i$-th edge of $\Pi_{\g,\tt}$, and by $\ell_i$ the length $|\g(\tt(i+1))-\g(\tt(i))|$ of the $i$-th edge of $\Pi_{\g,\tt}$, then by \eqref{first red} we have that, for some positive constant $C$,
\begin{equation}
\label{omega}
\sum_{i=1}^6x_i+\max_{1\le i\le 6}\Big|\ell_i-\frac{(12)^{1/4}}3\Big|<C\,\e\,,
\end{equation}
where of course $(12)^{1/4}/3$ is the length of one edge of a unit-area, regular hexagon. Since $(12)^{1/4}/3<1$, provided $\e$ is small enough and by \eqref{omega}, we deduce that $\max_i\ell_i<1$. In particular, using the scaling property \eqref{arc scaling} of $\arc_1$, and the fact that $\arc_1$ is increasing on $[0,\infty)$ (recall \eqref{implicit1}), we find that
\begin{eqnarray*}
  L(\g)\ge\sum_{i=1}^6\arc(\ell_i,x_i)=\sum_{i=1}^6\arc_1\Big(\frac{x_i}{\ell_i^2}\Big)
  \ge\sum_{i=1}^6\ell_i\,\arc_1\Big(\frac{x_i}{\ell_i}\Big)\,.
\end{eqnarray*}
By \eqref{omega}, up to further decreasing the value of $\e$, we can ensure that $\max_i x_i/\ell_i\le \pi/8$, and then combine the convexity of $\arc_1$ on $[0,\pi/8]$ (recall \eqref{implicit2}) with Jensen's inequality to conclude that
\begin{eqnarray}\label{chordal isoperimetric inequality}
  L(\g)\ge\sum_{i=1}^6\ell_i\,\arc_1\Big(\frac{x_i}{\ell_i}\Big)
  \ge\Big(\sum_{i=1}^6\,\ell_i\Big)\,\,\arc_1\bigg(\frac{\sum_{i=1}^6x_i}{\sum_{i=1}^6\,\ell_i}\bigg)
  =P(\Pi_{\g,\tt})\,\arc_1\Big(\frac{|E_\g\Delta\Pi_{\g,\tt}|}{P(\Pi_{\g,\tt})}\Big)\,.
\end{eqnarray}
(Inequality \eqref{chordal isoperimetric inequality} is strictly related to the {\it chordal isoperimetric inequality} found in \cite[Proposition 6.1-A]{hales} and \cite[15.5]{Morgan}, although it does not seem to exactly fit in those statements.) Now, by $\arc_1(0)=1$, $\arc_1'(0)=0$, and $\arc_1''(0)=12$ (recall \eqref{implicit1} and \eqref{implicit32}), we can find a positive constant $C$ such that
\begin{equation}
\label{arc coercivo}
\arc_1(x)\ge 1+6\,x^2-C\,|x|^3\,,\qquad\forall x\in[0,1/2]\,.
\end{equation}
Since we have $|E_\g\Delta\Pi_{\g,\tt}|\to 0$ and $P(\Pi_{\g,\tt})\to P(H_0)=2\,(12)^{1/4}$ in the limit $\e\to 0^+$, by further decreasing the value of $\e$, thanks to \eqref{omega}, we can ensure that $|E_\g\Delta\Pi_{\g,\tt}|<P(\Pi_{\g,\tt})/2$, and thus deduce from \eqref{chordal isoperimetric inequality}, \eqref{arc coercivo}, and \eqref{first red} that
\begin{equation}
\label{chordal isoperimetric inequality x}
L(\g)\ge P(\Pi_{\g,\tt})+\left(\frac{6}{2\,(12)^{1/4}}-C\,\e\right)\,|E_\g \Delta\Pi_{\g,\tt}|^2\,.
\end{equation}
For $\eta>0$ as in Theorem \ref{lem: d_hex bound}, up to further decreasing $\e$ so to entail that $\hd(\Pi_{\g,\tt},H_0)<\eta$, and thanks to the convexity of $\Pi_{\g,\tt}$, we can combine Theorem \ref{lem: d_hex bound} with \eqref{chordal isoperimetric inequality x} to find that, for some positive constant $C$,
\begin{equation}
\label{proof1}
L(\g)\ge2\,(12)^{1/4}\,\sqrt{|\Pi_{\g,\tt}|}+\frac{|\Pi_{\g,\tt}\Delta H_*|^2}{C}
+\left(\frac{3}{(12)^{1/4}}-C\,\e\right)\,|E_\g \Delta\Pi_{\g,\tt}|^2\,,
\end{equation}
where $H_*$ is a regular hexagon with $|H_*|=|\Pi_{\g,\tt}|$. Now, let us recall that, by \eqref{hales assumption 6}, $A(\g)>0$, so that, in particular, the orientation of $\g$ is such that $A(\g)=|E_\g|$. Denoting by $\circ$ the operation of concatenating curves, we see that
\[
\pi_{\g,\tt}=\ll\g(\tt(1)),\g(\tt(2))\rr\,\circ\,\ll\g(\tt(2)),\g(\tt(3))\rr\,\circ\,\cdots\,\circ\,\ll\g(\tt(6)),\g(\tt(1))\rr\,,
\]
defines a Lipschitz map $\pi_{\g,\tt}\in{\rm Lip}(\mathbb{S}^1;\R^2)$ that, thanks to \eqref{first red}, maps injectively $\mathbb{S}^1$ into $\pa\Pi_{\g,\tt}$, and in such a way that $A(\pi_{\g,\tt})=|\Pi_{\g,\tt}|$. By definition of $A$, $\pi_{\g,\tt}$, and of the secant oriented area functional $\a$ (recall \eqref{alfa}), we thus find that
\begin{eqnarray*}
|E_\g|-|\Pi_{\g,\tt}|&=&\sum_{i=1}^kA\big(\g|_{[\tt(i),\tt(i+1)]}\big)-\sum_{i=1}^kA\big(\ll\g(\tt(i)),\g(\tt(i+1))\rr\big)
\\
&=&\sum_{k=1}^6\a(\g,\tt(i),\tt(i+1))\,,
\end{eqnarray*}
so that, up to decreasing the value of $\e$, so to enforce from \eqref{first red} that $|\a(\g,\tt(i),\tt(i+1))|\le|E_\g\Delta\Pi_{\g,\tt}|<1/2$, we conclude that
\begin{equation}
  \label{important}
  |\Pi_{\g,\tt}|=|E_\g|-\s(\g,\tt)\,,
\end{equation}
Aiming at providing an efficient lower bound for the first term on the right-hand side of \eqref{proof1}, we use \eqref{important} to find that
\begin{eqnarray}\nonumber
  \sqrt{|\Pi_{\g,\tt}|}&=&\sqrt{|E_\g|}\,\sqrt{\frac{|\Pi_{\g,\tt}|}{|E_\g|}}
  =|E_\g|\,\sqrt{1-\frac{\s(\g,\tt)}{|E_\g|}}+\Big(\sqrt{|E_\g|}-|E_\g|\Big)\,\sqrt{\frac{|\Pi_{\g,\tt}|}{|E_\g|}}
  \\
  &=&
  |E_\g|\,\sqrt{1-\frac{\s(\g,\tt)}{|E_\g|}}+\frac{1-|E_\g|}{1+\sqrt{|E_\g|}}\,\,\sqrt{|\Pi_{\g,\tt}|}
  \\\label{easy}
  &\ge& |E_\g|\,\sqrt{1-\frac{\s(\g,\tt)}{|E_\g|}}+\frac{1-|E_\g|}{C}
\end{eqnarray}
where we have decreased $\e$ to ensure $|\Pi_{\g,\tt}|\ge 1/2$, and where $C$ is a positive constant. By means of the Taylor expansion we find that
\begin{equation}
  \label{radice}
  \sqrt{1-s}= 1-\frac{s}2-\frac{s^2}8+{\rm O}(s^3)\,,\qquad\mbox{as $s\to 0^+$}\,,
\end{equation}
and use \eqref{radice} with \eqref{easy} to bound from below the first and the third term on the right hand side of \eqref{proof1} as
\begin{eqnarray}\label{proof4}
&&2\,(12)^{1/4}\,\sqrt{|\Pi_{\g,\tt}|}+\Big(\frac3{(12)^{1/4}}-C\,\e\Big)\,|E_\g \Delta\Pi_{\g,\tt}|^2
\\\nonumber
&&\ge2\,(12)^{1/4}\,|E_\g|-(12)^{1/4}\,\s(\g,\tt)+\frac{1-|E_\g|}{C}
\\\nonumber
&&-2\,(12)^{1/4}\,\Big(\frac18-C\,\e\Big)\,\s(\g,\tt)^2+\Big(\frac3{(12)^{1/4}}-C\,\e\Big)
\,|E_\g \Delta\Pi_{\g,\tt}|^2\,.
\end{eqnarray}
Therefore, up to further decreasing the value of $\e$, we find that
\[
|\s(\g,\tt)|\le|E_\g\Delta\Pi_{\g,\tt}|\,,\qquad\mbox{and}\qquad \frac3{(12)^{1/4}}>\frac{2\,(12)^{1/4}}8\,,
\]
we finally conclude from \eqref{proof1} and \eqref{proof4} that
\[
L(\g)\ge2\,(12)^{1/4}\,|E_\g|+\frac{|\Pi_{\g,\tt}\Delta H_*|^2}{C}
-(12)^{1/4}\,\s(\g,\tt)+\frac{|E_\g\Delta\Pi_{\g,\tt}|^2+(1-|E_\g|)}{C}\,,
\]
that is
\begin{equation}
\label{fine davvero}
  C\,\de(\g,\tt)\ge |\Pi_{\g,\tt}\Delta H_*|^2+|E_\g\Delta\Pi_{\g,\tt}|^2+(1-|E_\g|)\,.
\end{equation}
Since $\d_{{\rm hex}}(E_\g)\le C\,(|\Pi_{\g,\tt} \Delta H^*|+|E_\g \Delta\Pi_{\g,\tt} |)$, we easily see that \eqref{fine davvero} implies \eqref{hales inequality improved hexagons}.
\end{proof}

\section{Proof of the quantitative Hales' hexagonal inequality, part two}\label{section hales intermed improv} Thanks to the argument presented in Section \ref{section proof of improved hales}, in order to complete the proof of Theorem \ref{thm quantita hales HI}, and thus of Theorem \ref{thm main}, we are left to prove the following theorem.

\begin{theorem}\label{thm slight improve}
There exist positive constants $a_1$ and $a_2$ with the following property. If $k\ge 2$, $(\gamma,\tt)\in  \Lip(\mathbb{S}^1;\R^2)\times \mathcal{P}_k$, and
\begin{equation}
  \label{k condition}
  A(\g)>0\quad\mbox{if $k=6$}\,,\qquad A(\gamma)\geq \frac1{100}\quad\mbox{if $2\le k\le 5$}\,,
\end{equation}
then
\begin{equation}
  \label{hales inequality x proof}
L(\g)+a_1\,(k-6)+(12)^{1/4}\,\s(\g,\tt)\ge 2\,(12)^{1/4}\,\min\{1,A(\g)\}+a_2\,|k-6|\,,
\end{equation}
with equality if and only if $(\gamma,\tt)$ corresponds to a unit-area regular hexagon.
\end{theorem}

Theorem \ref{thm slight improve} improves on Hales' hexagonal isoperimetric inequality (Theorem \ref{thm hales HI}) since it requires no lower bound on $A(\g)$ when $k\ge 6$ and weaker ones when $2\le k\le 5$, and since, when $k\ne 7$, it provides the additional lower bound $a_2\,|k-6|$. These improvements are obtained by carefully revisiting Hales' original argument, which is why the methods of this section overlap with those of \cite{hales}.

\begin{proof}
[Proof of Theorem \ref{thm slight improve}] We make the following claim:

\bigskip

\noindent {\it Claim}: There are positive constants $a$ and $c$ such that, if $k\ge 2$, $k\ne 6$, and $(\gamma,\tt)\in  \Lip(\mathbb{S}^1;\R^2)\times \mathcal{P}_k$ satisfies \eqref{k condition}, then
\begin{equation}
\label{equationTemp}
L(\gamma)+a\,(k-6)+(12)^{1/4}\, \sigma(\gamma, \t) \geq 2(12)^{1/4}\,\min\{1,A(\gamma)\} + c\,;
\end{equation}
and, if $(\gamma,\tt)\in  \Lip(\mathbb{S}^1;\R^2)\times \mathcal{P}_6$ with $1/5\ge A(\g)>0$, then
\begin{equation}
\label{equationTemp 6}
L(\gamma)+(12)^{1/4}\, \sigma(\gamma, \t) > 2(12)^{1/4}\,\min\{1,A(\gamma)\}\,.
\end{equation}

\bigskip

\noindent {\it The claim implies the theorem}: Indeed, when $k=6$, Theorem \ref{thm hales HI} implies that
\[
L(\gamma)+(12)^{1/4}\, \sigma(\gamma, \t) \ge 2(12)^{1/4}\,\min\{1,A(\gamma)\}
\]
whenever $(\gamma,\tt)\in  \Lip(\mathbb{S}^1;\R^2)\times \mathcal{P}_6$ and $A(\g)\ge 2\,\pi/(\sqrt{3}\,(6^2))$, and with equality if and only if $(\g,\tt)$ corresponds to a regular unit-area hexagon. Since $2\,\pi/(\sqrt{3}\,(6^2))< 1/5$, the combination of \eqref{equationTemp 6} with Theorem \ref{thm hales HI} proves the case $k=6$ of Theorem \ref{thm slight improve}. When $k\ge 2$, $k\ne 6$, we set
\[
a_1=a_2+a\,,\qquad a_2=\frac{c}8\,,
\]
and deduce \eqref{hales inequality x proof} from \eqref{equationTemp} as follows: \eqref{hales inequality x proof} with $k>6$ is equivalent to
\[
L(\gamma)+(a_1-a_2)\,(k-6)+(12)^{1/4}\,\s(\gamma, {\bf t})\ge 2\,(12)^{1/4}\,\min\{1,A(\gamma)\}\,,
\]
which is implied by \eqref{equationTemp} thanks to $a_1-a_2=a$; if, instead, $2\leq k \leq 5$, then by \eqref{equationTemp} we get
\begin{align*}
L(\gamma) +(12)^{1/4}\,\sigma(\gamma, {\bf t}) &\geq 2(12)^{1/4}\,\min\{1, A(\gamma)\}+c-a\,(k-6)
\\
&= 2(12)^{1/4}\,\min\{1, A(\gamma)\}+8\,a_2+(a_2-a_1)\,(k-6)\,,
\\
&\ge 2(12)^{1/4}\,\min\{1, A(\gamma)\}-(a_1+a_2)\,(k-6)\,,
\end{align*}
that is \eqref{hales inequality x proof}.

\bigskip

We can thus focus on the proof of the above claim. It will be convenient to set some notation. For every $(\g,\tt)\in\Lip(\mathbb{S}^1;\R^2)\times\P_k$ we set for brevity
\[
\e(\g,\tt)=a\,(k-6)+(12)^{1/4}\, \sigma(\gamma, \t)-2(12)^{1/4}\,\min\{1,A(\gamma)\}\,,
\]
where $a>0$ is a constant to be determined. Moreover, dropping the dependency on $\tt$ in the interest of brevity, we set, for each $i=1,...,k+1$,
\begin{eqnarray}\label{def gammai}
\g_i&=&\g\big|_{[\tt(i),\tt(i+1)]}\,,
\\\label{def Ai}
A_i&=&A(\g_i)+A\big(\ll\g(\tt(i+1)),\g(\tt(i))\rr\big)\,,
\end{eqnarray}
so that
\[
L(\g)=\sum_{i=1}^kL(\g_i)\,,\qquad \s(\g,\tt)=\sum_{i=1}^k\max\Big\{-\frac12,\min\Big\{A_i,\frac12\Big\}\Big\}\,.
\]
We now present a series of reduction steps.

\bigskip

\noindent {\it Step one}: For each $(\g,\tt)\in\Lip(\mathbb{S}^1;\R^2)\times\P_k$, there is $\bar\g\in\Lip(\mathbb{S}^1;\R^2)$ such that, for each $i=1,...,k$, the restriction of $\bar\g$ to the interval $[\tt(i),\tt(i+1)]$ is the unit speed parametrization of a circular arc with $\bar\g(\tt(i))=\g(\tt(i))$, and moreover
\begin{equation}
  \label{reduction gammabar}
  L(\g)+\e(\g,\tt)\ge L(\bar\g)+\e(\bar\g,\tt)\,,
\end{equation}
for every choice of $a>0$. (Notice carefully that neither $\g$ or $\bar\g$ are assumed or required to be injective.)

\medskip

To prove this, it is enough to define $\bar\g$ on $[\tt(i),\tt(i+1)]$, $1\le i\le k$, as the unit speed parametrization of a circular arc with endpoints $\bar\g(\tt(i))=\g(\tt(i))$ and $\bar\g(\tt(i+1))=\g(\tt(i+1))$ and such that
\[
A(\bar\g_i)=A(\g_i)\,,
\]
(where $\bar\g_i$ is the restriction of $\bar\g$ to $[\tt(i),\tt(i+1)]$, see \eqref{def gammai}). Then, by construction, $A(\bar\g)=A(\g)$ and $\s(\bar\g,\tt)=\s(\g,\tt)$, while, by Dido's inequality, $L(\bar\g)\le L(\g)$, so that \eqref{reduction gammabar} holds.

\medskip

Having proved step one, for the rest of the proof we will be able to work directly with pairs
\[
(\g,\tt)\in\Cir_k\,,
\]
as defined by the condition that, for each $i=1,...,k$, the restriction of $\g$ to the interval $[\tt(i),\tt(i+1)]$ is the unit speed parametrization of a circular arc.

\bigskip

\noindent {\it Step two}: We prove that, for each $(\g,\tt)\in\Cir_k$, there is $\bar\g$ such that $(\bar\g,\tt)\in\Cir_k$,
\[
\bar{A}_i\ge-\frac12\,,\qquad\forall i=1,...,k\,,
\]
and \eqref{reduction gammabar} holds (for any choice of $a>0$). Here $\bar{A}_i$ is defined from $\bar\g$ as $A_i$ was defined from $\g$ in \eqref{def Ai}.

\medskip

To prove this, we set $\bar\g=\g$ on those intervals $[\tt(i),\tt(i+1)]$ corresponding to $i=1,...,k$ such that $A_i\ge-1/2$; if, instead, $A_i<-1/2$, then we define $\bar\g$ on $[\tt(i),\tt(i+1)]$ as a unit-speed parametrization of a circular arc with endpoints $\g(\tt(i))$ and $\g(\tt(i+1))$ such that $\bar A_i=-1/2$: since $\bar\g$ and $\g$ have the same endpoints and $A_i<-1/2$, we see that this implies
\[
A(\bar\g_i)>A(\g_i)\,,\qquad L(\bar\g_i\big)\le L(\g_i)\,.
\]
In particular, $L(\g)\ge L(\bar\g)$, $A(\bar\g)\ge A(\g)$, and $\s(\bar\g,\tt)=\s(\g,\tt)$, which implies the validity of \eqref{reduction gammabar}.

\medskip

Thanks to step one and step two we have reduced to the case when
\begin{equation}
  \label{reduction second}
  (\g,\tt)\in\Cir_k\,,\qquad A_i\ge -\frac12\,,\quad\forall i=1,...,k\,.
\end{equation}

\bigskip

\noindent {\it Step three}: We prove that given $(\g,\tt)$ as in \eqref{reduction second} there is $(\bar\g,\tt)\in\Cir_k$ such that \eqref{reduction gammabar} holds (for any value of $a$), and such that all the $\bar{A}_i$'s in $[-1/2,1/2]$ have the same sign, that is, setting
\[
\bar{I}=\Big\{i=1,...,k: |\bar{A}_i|\le\frac12\Big\}\,,
\]
we have that either $\bar{A}_i\ge0$ for all $i\in\bar{I}$ or $\bar{A}_i\le0$ for all $i\in\bar{I}$.

\medskip

To prove this, let us consider the set $K$ of those pairs $(i,j)$ with $1\le i,j\le k$, $i\ne j$, and
\[
\frac12\ge A_i>0>A_j\ge-\frac12\,,
\]
which we can assume to be non-empty, otherwise, we have nothing to prove. Let $(i,j)\in K$ and let $\beta=A_{i}+A_{j}$. Let us assume for a moment that $\beta\ge0$. In this case, we define $\g^*$ by taking $\g^*=\g$ everywhere on $\mathbb{S}^1$ except on: $[\tt(j),\tt(j+1)]$, on which we define $\g^*$ as the unit speed parametrization of the segment from $\g(\tt(j))$ to $\g(\tt(j+1))$ (so that $A^*_{j}=0$ provided $A^*_j$ is defined from $\g^*$ in analogy to \eqref{def Ai}); and $[\tt(i),\tt(i+1)]$, on which we define $\g^*$ as the unit speed parametrization of a circular arc from $\g(\tt(i))$ to $\g(\tt(i+1))$ such that $A^*_{i}=\beta\ge0$. In this way $A(\g^*)=A(\g)$ and $\s(\g^*,\tt)=\s(\g,\tt)$ (recall the restriction $1/2\ge |A_i|,|A_j|$). Moreover, $L(\g_{j}^*)\le L(\g_{j})$ (as on $[\tt(j),\tt(j+1)]$ we have replaced a circular arc with its chord segment), and $L(\g^*_{i})\le L(\g_{i})$ (since $A^*_{i}=\beta\le A_{i}$, on $[\tt(i),\tt(i+1)]$ we have modified the curvature of a circular arc with fixed endpoints so to decrease the amount of enclosed area), so that, in total, $L(\g^*)\le L(\g)$.

\medskip

In summary, when $\beta\ge0$ we have constructed a curve $\g^*$ such that $L(\g^*)+\e(\g^*,\tt)\le L(\g)+\e(\g,\tt)$, $A^*_{i}=\beta$, $A^*_{j}=0$, and $\g^*=\g$ on $[\tt(h),\tt(h+1)]$ for all $h\ne i,j$. In the case $\beta<0$, an analogous construction produces $\g^*$ such that $L(\g^*)+\e(\g^*,\tt)\le L(\g)+\e(\g,\tt)$, $A^*_{i}=0$, $A^*_{j}=\beta$,  and $\g^*=\g$ on $[\tt(h),\tt(h+1)]$ for all $h\ne i,j$. By iterating this construction finitely many times we end up constructing a curve $\bar\g$ satisfying \eqref{reduction gammabar} and such that, with $I=\{i:|\bar{A}_i|\le1/2\}$, we have that either $\bar{A}_i \geq 0$ or $\bar{A}_i \leq 0$ for all $i\in \bar{I}$.

\medskip

Therefore, we have so far reduced to prove that \eqref{equationTemp} holds when $k\ge 2$, $k\ne 6$,
\begin{equation}
  \label{reduction second x}
  (\g,\tt)\in\Cir_k\,,\qquad A_i\ge -\frac12\,,\quad\forall i=1,...,k\,,
\end{equation}
and, setting $I=\{i:|A_i|\le 1/2\}$, we have that
\begin{equation}
  \label{reduction third}
  \mbox{either $A_i\ge0$ for all $i\in I$, or $A_i\le 0$ for all $i\in I$}\,.
\end{equation}

\bigskip

\noindent {\it Step four}: In this step we impose the first restriction on the constant $a>0$ appearing in the definition of $\e(\g,\tt)$, that is, we impose
\begin{equation}
\label{restriction a 1}
a < \frac{\sqrt{\pi}}{2}-\frac{3}{8}(12)^{1/4}\,.
\end{equation}
With this restriction on $a$, we prove that if $(\g,\tt)\in\Cir_k$ satisfies \eqref{reduction second x}, \eqref{reduction third}, and is such that $A_i>1/2$ for some $1\le i\le k$, then
\begin{equation}
  \label{a fine}
  L(\g)+\e(\g,\tt)\geq \min\Big\{\frac{7}{10},c(a)\Big\}\,,
\end{equation}
where
\begin{equation}
  \label{def of ca}
  c(a)=2\,\sqrt\pi-(3/2)\,(12)^{1/4}-4\,a\,,
\end{equation}
is strictly positive thanks to the strict sign in \eqref{restriction a 1}. We notice that this step, combined with the previous ones, effectively reduces the proof of \eqref{equationTemp} (thus of the theorem) to the case when $k\ge 2$, $k\ne 6$, $(\g,\tt)\in\Cir_k$ and
\begin{eqnarray}
\label{reduction four}
\mbox{either}\qquad \frac12\ge A_i\ge0\quad\forall i\,,\qquad\mbox{or}\qquad 0\ge A_i\ge-\frac12\quad\forall i\,.
\end{eqnarray}
We now turn to the proof of \eqref{a fine}:

\medskip

Let $(\g,\tt)\in\Cir_k$ satisfy \eqref{reduction second}, \eqref{reduction third} and be such that $A_{i_0}>1/2$ for some $i_0$. Set
\[
I = \Big\{i: A_i > \frac{1}{2}\Big\}\,,\qquad  J =\Big\{i:|A_i|\le\frac12\Big\}\,,
\]
so that $\{1,...,k\}=I\cup J$, $I\cap J=\varnothing$, $I\ne\varnothing$, and, setting
\[
J^+=\{i\in J:A_i\ge0\}\,,\qquad J^-=\{i\in J:A_i\le0\}\,,
\]
and recalling \eqref{reduction third}, either $J=J^+$ or $J=J^-$. We address the two cases by different arguments.

\medskip

\noindent {\it Case one, $J=J^+$}: In this case by $I\ne\varnothing$ we have
\[
\s(\g,\tt)=\frac{\#\,I}2+\sum_{i\in J}A_i\ge \frac12\,.
\]
Hence, by the trivial bound $a\,(k-6)\ge -4\,a$ (recall, $k\ge 2$), and by the isoperimetric inequality $L(\g)\ge 2\,\sqrt\pi\,\sqrt{A(\g)}$ (recall that $A(\g)\ge0$ by \eqref{k condition}), we find
\begin{eqnarray*}
  L(\g)+\e(\g,\tt)&\ge&2\,\sqrt\pi\,\sqrt{A(\g)}-4\,a+\frac{(12)^{1/4}}2 -2\,(12)^{1/4}\,\min\{1,A(\gamma)\}
  \\
  &\ge&2\,\sqrt\pi-4\,a+\frac{(12)^{1/4}}2 -2(12)^{1/4}=\Big(2\,\sqrt\pi-\frac32\,(12)^{1/4}\Big)-4\,a\,,
\end{eqnarray*}
where in the last step we have used the fact that $x\ge0\mapsto 2\,\sqrt\pi\, \sqrt{x}-2\,(12)^{1/4}\,\min\{1,x\}$ achieves its minimum at $x=1$. This proves \eqref{a fine} when $J=J^+$.

\medskip

\noindent {\it Case two, $J=J^-$}: Setting for brevity
\[
x=A(\g)\,,\qquad z = \sum_{i \in J}(-A_i)\,,
\]
we have $x,z\ge 0$, and, thanks to $I\ne\varnothing$,
\begin{equation}
\label{eq sigma bound}
  \s(\g,\tt)=\frac{\# I}{2}+\sum_{j\in J}A_j\ge \frac{1}{2}-z\,.
\end{equation}
Now, by Dido's inequality, we have that
\begin{equation}
  \label{usedido}
  L(\g_i)\ge \sqrt{2\,\pi\,|A_i|}\,,\qquad\forall i=1,...,k\,,
\end{equation}
so that $A_i> 1/2$ for $i\in I$ implies that
\begin{eqnarray*}
  L(\g)=\sum_{i=1}^k L(\g_i)\ge\sqrt{2\,\pi}\,\sum_{i=1}^k\sqrt{|A_i|}
  \ge\sqrt{\pi}\,\#I +\sqrt{2\,\pi}\,\sum_{i\in J}\sqrt{|A_i|}\,.
\end{eqnarray*}
Since $|A_i|\le 1/2$ for $i\in J$ and $\sqrt2\,|t|\le\sqrt{|t|}$ for $|t|\le1/2$ we conclude that
\begin{equation}
  \label{ma}
  L(\g)\ge\sqrt{\pi}\,\#I+2\,\sqrt\pi\,\sum_{i\in J}|A_i|\ge \sqrt\pi+2\,\sqrt\pi\,z\,,
\end{equation}
taking into account $I\ne\varnothing$ and the definition of $z$. By using as done in the other case that $a\,(k-6)\ge-4\,a$, we thus obtain a first lower bound on $L(\g)+\e(\g,\tt)$, namely
\begin{eqnarray}\nonumber
  L(\g)+\e(\g,\tt)&=&L(\g)+a\,(k-6)+(12)^{1/4}\s(\g,\tt)-2\,(12)^{1/4}\,\min\{1,A(\g)\}
  \\\nonumber
  &\ge&
  \sqrt\pi+2\,\sqrt\pi\,z-4\,a+(12)^{1/4}\Big(\frac{1}{2}-z\Big)-2\,(12)^{1/4}
  \\\label{eq bound 1}
  &\ge&
  \Big(2\,\sqrt\pi-(12)^{1/4}\Big)\,z+\Big(\sqrt\pi-\frac{3}2\,(12)^{1/4}\Big)-4\,a\,.
\end{eqnarray}
This lower bound will be sufficient to prove \eqref{a fine} only for certain values of $z$. For this reason, before discussing the latter point, we obtain a second, complementary, lower bound, that combined with \eqref{eq bound 1} will allow us to deduce \eqref{a fine}. To obtain this second lower bound, consider $\bar{\gamma}$ obtained from $\gamma$ by ``reflecting'' with respect to their chords all the circular arcs $\gamma_i$ corresponding to $i\in J$. In this way $L(\bar\g)=L(\g)$, while $A(\bar\g_i)=-A(\g_i)$ for all $i\in J$, gives
\begin{equation}\label{eq: area change}
A(\bar{\g})=\sum_{i\in I}A(\g_i)-\sum_{i\in J}A(\g_i)=A(\g)- 2\,\sum_{i\in J}A_i\,=x+2\,z\,.
\end{equation}
By applying the isoperimetric inequality to $\bar\g$ (notice that $A(\bar\g)\ge A(\g)\ge 0$) and using \eqref{eq: area change} we thus find
\begin{equation}
  \label{as detailed}
  L(\gamma)=L(\bar{\gamma})\ge 2\,\sqrt\pi\,\sqrt{A(\bar\g)} \geq 2\,\sqrt{\pi}\,\sqrt{x+2\,z}\,,
\end{equation}
which, combined with \eqref{eq sigma bound}, gives
\begin{eqnarray}\nonumber
  L(\g)+\e(\g,\tt)&=&L(\g)+a\,(k-6)+(12)^{1/4}\s(\g,\tt)-2\,(12)^{1/4}\,\min\{1,A(\g)\}
  \\\nonumber
  &\ge&
  2\,\sqrt{\pi}\,\sqrt{x+2\,z}-4\,a+(12)^{1/4}\Big(\frac{1}{2}-z\Big)-2\,(12)^{1/4}\,\min\{1,x\}\,,
  \\\label{eq bound 2}
  &\ge&
  2\,\sqrt{\pi}\,\sqrt{1+2\,z}-4\,a+(12)^{1/4}\Big(\frac{1}{2}-z\Big)-2\,(12)^{1/4}\,,
\end{eqnarray}
where we have used $x\ge0\mapsto 2\,\sqrt\pi\, \sqrt{x+2\,z}-2\,(12)^{1/4}\,\min\{1,x\}$ has a minimum at $x=1$. In summary, setting
\begin{eqnarray*}
  f(z)&=&  \Big(2\,\sqrt\pi-(12)^{1/4}\Big)\,z+\Big(\sqrt\pi-\frac{3}2\,(12)^{1/4}\Big)-4\,a\,,
  \\
  g(z)&=&2\,\sqrt{\pi}\,\sqrt{1+2\,z}-(12)^{1/4}\,z-\frac32\,(12)^{1/4}-4\,a\,,
\end{eqnarray*}
and $h=\max\{f,g\}$, we are left to prove that
\[
\inf_{z\ge 0}h(z)=\min\Big\{\frac7{10},c(a)\Big\}\,.
\]
Since $f$ is an affine, increasing function, $g$ is concave on $[0,\infty)$, and $f(3/2)=g(3/2)$, we have that $h=g$ on $[0,3/2]$ and $h=f$ on $[3/2,\infty)$. In particular, if $z\ge 3/2$, then, recalling \eqref{restriction a 1}, we find
\[
h(z)=f(z)\ge f(3/2)=4\,\sqrt{\pi}-3\,(12)^{1/4}-4\,a\ge2\,\sqrt{\pi}-\frac32\,(12)^{1/4}\ge\frac{7}{10}\,.
\]
If, instead, $z\in[0,3/2]$, then $h(z)\ge g(z)\ge g(0)=c(a)$ by definition \eqref{def of ca} of $c(a)$.

\bigskip

\noindent {\it Step five}: We conclude the proof of the theorem. Based on the previous four steps, we have to prove the following reduced version of our opening claim:

\medskip

\noindent {\it Reduced Claim}: There are positive constants $a$ and $c$ such that the following holds. Let $k\ge 2$ and $(\gamma,\tt)\in\Cir_k$ satisfy
\begin{equation}
  \label{k condition xx}
  A(\gamma)\geq \frac1{100}\,,\qquad\mbox{if $2\le k\le 5$}\,,
\end{equation}
and let either $0\le A_i\le 1/2$ for all $i=1,...,k$ or $0\ge A_i\ge -1/2$ for all $i=1,...,k$. Then, when $k\ne 6$,
\begin{equation}
\label{equationTemp x}
L(\gamma)+a\,(k-6)+(12)^{1/4}\, \sigma(\gamma, \t) \geq 2(12)^{1/4}\,\min\{1,A(\gamma)\} + c\,;
\end{equation}
and, when $k=6$ and $1/5\ge A(\g)>0$,
\begin{equation}
\label{equationTemp 6 x}
L(\gamma)+(12)^{1/4}\, \sigma(\gamma, \t) > 2(12)^{1/4}\,\min\{1,A(\gamma)\}\,.
\end{equation}

\medskip

To begin the proof of this reduced claim, let us recall that, so far, we have only imposed on $a$ the constraint \eqref{restriction a 1}. Since $(\sqrt\pi/2)-(3/8)\,(12)^{1/4}> 1/10$, we can work with any $a\le 1/10$. The choice
\begin{equation}
  \label{restriction a 2}
  a=\frac{3}{50}\,,
\end{equation}
is enforced from now on, for the sake of definiteness.

\medskip

Next, we notice that the case $k=2$ is easily dealt with. Indeed, in this case, $|A_i|\le 1/2$ implies $\s(\g,\tt)=A_1+A_2=A(\g)$ as well as $\min\{A(\g),1\}=A(\g$): therefore, by applying the isoperimetric inequality we find that
\[
L(\g)+\e(\g,\tt)\ge 2\,\sqrt{\pi\,A(\g)}-4\,a-(12)^{1/4}\,A(\g)\ge\min\Big\{h(1),h\Big(\frac1{100}\Big)\Big\}>\frac{9}{100}\,,
\]
where we have used that $h(x)=2\,\sqrt{\pi\,x}-(12)^{1/4}\,x-(6/25)$ is concave on $[0,1]$ and that $1\ge A(\g)\ge 1/100$.

\medskip

When $k\ge 3$ there is no immediate relation between $x=A(\g)$ and $y=\s(\g,\tt)=\sum_{i=1}^k\,A_i$ that we can use. To discuss this case it is convenient to collect the following four lower bounds:
\begin{eqnarray}
\label{Isoperimetric}
L(\gamma) &\geq& 2\,\sqrt{\pi\, x}\,,
\\
\label{Negative Arcs}
L(\gamma) &\geq& 2\,\sqrt{\pi\,(x+2\,y^-)}\,,
\\
\label{Dido}
L(\gamma) &\geq& 2\,\sqrt{\pi}\,|y|\,,
\\
\label{Polygonal}
L(\gamma) &\geq& 2\,\sqrt{k\,\tan\Big(\frac{\pi}{k}\Big)}\,\sqrt{(x-y)^+}\,,
\end{eqnarray}
where $z^+=\max\{z,0\}$ and $z^-=\max\{-z,0\}$. Of course \eqref{Isoperimetric} is just the isoperimetric inequality, as already used repeatedly, while \eqref{Negative Arcs} follows from the isoperimetric inequality  by the argument used in proving \eqref{as detailed}. Concerning \eqref{Dido}, by Dido's inequality $L(\g_i)\ge \sqrt{2\,\pi\,|A_i|}$ with $|A_i|\le 1/2$ for each $i$ we obtain
\[
L(\g)\ge\sqrt{2\,\pi}\,\sum_{i=1}^k\sqrt{|A_i|}\ge 2\,\sqrt{\pi}\,\sum_{i=1}^k|A_i|=2\,\sqrt\pi\,|\s(\g,\tt)|\,,
\]
where in the last identity we have taken into account that, under our assumptions on $\g$, all the $A_i$'s have the same sign. Finally, let us define $\pi_\g\in\Lip(\mathbb{S}^1;\R^2)$ by taking
\[
\pi_\g=\ll\g(\tt(1)),\g(\tt(2))\rr\,\circ\,\ll\g(\tt(2)),\g(\tt(3))\rr\,\circ\,\cdots\,\circ\,\ll\g(\tt(k-1)),\g(\tt(k))\rr\,\circ\,\ll\g(\tt(k)),\g(\tt(1))\rr\,.
\]
Clearly $L(\g)\ge L(\pi_\g)$ and $\pi_\g$ is a polygonal curve with $k$ vertexes. By the polygonal isoperimetric inequality {\it for immersed curves}\footnote{The point here is that $\pi_\g$ may not be injective, therefore we are not able to bound $L(\pi_\g)$ from below by directly using the polygonal isoperimetric inequality {\it for sets}, namely \eqref{isoperimetric inequality kgons}!} we have
\begin{equation}
  \label{see later}
  L(\pi_\g)^2\ge 4\,k\,\tan\Big(\frac\pi{k}\Big)\,A(\pi_\g)=4\,k\,\tan\Big(\frac\pi{k}\Big)\,\Big(A(\g)-\s(\g,\tt)\Big)\,,
\end{equation}
where in asserting $A(\pi_\g)=A(\g)-\s(\g,\tt)$ we have used $|A_i|\le 1/2$ for all $i=1,...,k$.

\medskip

Based on \eqref{Isoperimetric}, \eqref{Negative Arcs}, \eqref{Dido}, and \eqref{Polygonal}, we now set
\begin{eqnarray*}
g_k(x,y)&=&a\,(k-6)+(12)^{1/4}\,\Big(y-2\,\min\{1,x\}\Big)\,,
\\
f_{k,1}(x,y)&=&g_k(x,y)+2\,\sqrt{\pi x}\,,
\\
f_{k,2}(x,y)&=&g_k(x,y)+2\,\sqrt{\pi\,(x+2\,y^-)}\,,
\\
f_{k,3}(x,y)&=&g_k(x,y)+2\,\sqrt{\pi}\,|y|\,,
\\
f_{k,4}(x,y)&=&g_k(x,y)+2\,\sqrt{k\,\tan\Big(\frac{\pi}{k}\Big)}\,\sqrt{(x-y)^+}\,.
\end{eqnarray*}
Taking into account that each $f_{k,i}(\cdot,y)$ is increasing on $x\ge 1$, and setting
\[
f_k=\max\{f_{k,1},f_{k,2},f_{k,3},f_{k,4}\}\,,
\]
we can conclude the proof of the reduced claim, and thus of the theorem, by showing that
\begin{eqnarray}
\label{yokai2}
  &&\inf_{k\ge 7}\,\inf_{[0,1]\times\R}f_k>0\,,
  \\
  \label{yokai2 k6}
  &&\inf_{\R}f_6(x,\cdot)>0\,,\qquad\forall x\in\Big(0,\frac15\Big]\,,
  \\
  \label{yokai2 min5}
  &&\inf_{3\le k\le 5}\,\inf_{[1/100,1]\times\R}f_k>0\,.
\end{eqnarray}
To prove this we will combine different lower bounds on each the $f_{k,i}$'s.

\medskip

\noindent {\it First lower bound}: We prove the existence of a positive (computable) constant $c_0$ such that, for all $k\ge3$,
\begin{equation}
  \label{lower bound one}
  f_k(x,y)\ge c_0\,,\qquad\forall (x,y)\in[0,1]\times\Big\{\Big(-\infty,-\frac52\,\Big]\cup\Big[\,\frac45,\infty\Big)\Big\}\,.
\end{equation}
Indeed, if $k\ge 3$, $x\in[0,1]$, and $y\in\R$, then
\begin{eqnarray*}
f_k(x,y)&\ge&f_{k,3}(x,y)=2\,\sqrt{\pi}\,|y|+a\,(k-6)+(12)^{1/4}\,y-2\,(12)^{1/4}\,x
\\
&\ge&2\,\sqrt\pi\,|y|+(12)^{1/4}\,y-2\,(12)^{1/4}-3\,a
\end{eqnarray*}
This last function is non-negative if and only if (recall that $a=3/50$)
\[
\mbox{either}\quad y\ge \frac{2\,(12)^{1/4}+(9/50)}{2\,\sqrt\pi+(12)^{1/4}}\,,
\qquad\mbox{or}\quad y\le -\frac{2\,(12)^{1/4}+(9/50)}{2\,\sqrt\pi-(12)^{1/4}}\,.
\]
Taking into account that
\[
\frac{2\,(12)^{1/4}+(9/50)}{2\,\sqrt\pi+(12)^{1/4}}<\frac4{5}\,,\qquad \frac{2\,(12)^{1/4}+(9/50)}{2\,\sqrt\pi-(12)^{1/4}}<\frac52\,,
\]
we conclude that $2\,\sqrt\pi\,|y|+(12)^{1/4}\,y-2\,(12)^{1/4}-3\,a$ is uniformly positive when either $y\le -5/2$ or $y\ge 4/5$.

\medskip

\noindent {\it Second lower bound}: We prove that, for all $k\ge 3$,
\begin{equation}
  \label{lower bound two}
  f_k(x,y)\ge f_{k,1}(x,0)\qquad\forall x\in[0,1]\,, y\ge-\frac52\,.
\end{equation}
Since $f_k\ge f_{k,1}$, this is obvious from the definition of $f_{k,1}$ when $y\ge0$ (one just drops the term $(12)^{1/4}\,y$). Assuming now that $y\in[-5/2,0]$ we notice that
\begin{eqnarray*}
  f_k(x,y)&\ge&f_{k,2}(x,y)=2\,\sqrt{\pi\,(x+2\,y^-)}+a\,(k-6)+(12)^{1/4}\,y-2\,(12)^{1/4}\,x
  \\
  &=&2\,\sqrt{\pi\,(x+2\,|y|)}-(12)^{1/4}\,|y|+a\,(k-6)-2\,(12)^{1/4}\,x
  \\
  &\ge&2\,\sqrt{\pi\,x}+a\,(k-6)-2\,(12)^{1/4}\,x=f_{k,1}(x,0)\,,
\end{eqnarray*}
where in the last step we have used that
\begin{equation}
  \label{interesting}
  F(x,b)=2\,\sqrt{\pi\,(x+2\,b)}-(12)^{1/4}\,b-2\,\sqrt{\pi\,x}\ge0\,,\qquad\forall (x,b)\in[0,1]\times\Big[0,\frac52\Big]\,.
\end{equation}
To prove \eqref{interesting} we just notice that, for every $x\ge0$, $F(x,\cdot)$ is concave on $[0,\infty)$, so that
\[
\inf_{[0,5/2]}F(x,\cdot)\ge \min\Big\{F(x,0),F\Big(x,\frac52\Big)\Big\}=\min\Big\{0,F\Big(x,\frac52\Big)\Big\}\,.
\]
Now, for every $x\in(0,1]$ we have
\begin{eqnarray*}
&&F\Big(x,\frac52\Big)=2\,\sqrt{\pi}\,\sqrt{x+5}-\frac{5\,(12)^{1/4}}2-2\,\sqrt{\pi\,x}\,,
\\
&&\frac{\pa F}{\pa x}\Big(x,\frac52\Big)=\sqrt{\frac{\pi}{x+5}}-\sqrt{\frac{\pi}x}<0\,,
\end{eqnarray*}
so that
\[
\inf_{0\le x\le 1}F\Big(x,\frac52\Big)\ge F\Big(1,\frac52\Big)
=2\,\sqrt{6\,\pi} -\frac{5\,(12)^{1/4}}2-2\,\sqrt{\pi}>\frac25\,.
\]
This proves \eqref{interesting}, and thus \eqref{lower bound two}.

\medskip

\noindent {\it Third lower bound}: Motivated by \eqref{lower bound two}, we prove the existence of a (computable) positive constant $c_1$ such that
\begin{eqnarray}
  \label{fk1 1}
  f_{k,1}(x,0)\!\!&\ge&\!\! c_1\,,\qquad\forall x\in[0,1]\,,\qquad\mbox{if $k\ge 9$}\,,
  \\
  \label{fk1 1.5}
  f_{k,1}(x,0)\!\!&\ge&\!\! c_1\,,\qquad\forall x\in\Big[0,\frac9{10}\Big]\,,\qquad\mbox{if $7\le k\le 8$}\,,
  \\
  \label{fk1 2}
  f_{k,1}(x,0)\!\!&\ge&\!\! c_1\,,\qquad\forall x\in\Big[\frac1{100},\frac45\Big]\,,\qquad\mbox{if $3\le k\le 5$}\,,
\end{eqnarray}
and also show that
\begin{equation}
  \label{f61 1}
  f_{6,1}(x,0)>0\,,\qquad\,\,\forall x\in\Big(0,\frac9{10}\Big]\,.
\end{equation}
To prove this, let us notice that
\[
f_{k,1}(x,0)=2\,\sqrt{\pi x}+a\,(k-6)-2\,(12)^{1/4}\,x=-q_k(\sqrt{x})\,,
\]
where
\[
q_k(t)=2\,(12)^{1/4}\,t^2-2\,\sqrt{\pi}\,t-a\,(k-6)\,.
\]
The roots of $q_k$ are given by
\[
t_k^\pm=\frac{2\,\sqrt{\pi}\pm\sqrt{4\,\pi+4\,a\,(k-6)\,2\,(12)^{1/4}}}{4\,(12)^{1/4}}\,.
\]
Thanks to $a=3/50$ and $k\ge3$ we have
\[
4\,\pi+4\,a\,(k-6)\,2\,(12)^{1/4}\ge4\,\pi-\frac{36}{25}\,(12)^{1/4}> 9\,.
\]
Therefore both $t_k^\pm$ are real, with
\begin{eqnarray*}
&&t_k^-<0<t_k^+\,,\qquad\mbox{if $k\ge 7$}\,,
\\
&&t_6^-=0<t_6^+\,,
\\
&&0<t_k^-<t_k^+\,,\qquad\mbox{if $3\le k\le 5$}\,.
\end{eqnarray*}
In particular, if $k\ge 6$, then we have
\begin{eqnarray}\label{xka}
&&\big\{x\ge 0: f_{k,1}(x,0)\ge0\big\}=[0,z_k]\,,
\\\nonumber
&&\mbox{where}\quad
z_k:=(t_k^+)^2=\frac{\pi+a\,(k-6)\,(12)^{1/4}+\sqrt{\pi}\,\sqrt{\pi+a\,(k-6)\,2\,(12)^{1/4}}}{2\,(12)^{1/2}}\,,
\end{eqnarray}
while if $3\le k\le 5$, then (with $z_k$ as in \eqref{xka})
\begin{eqnarray}\label{xka k35}
&&\big\{x\ge0:f_{k,1}(x,0)\ge0\big\}=[y_k,z_k]\,,
\\\nonumber
&&\mbox{where}\quad
y_k:=(t_k^-)^2=\frac{\pi+a\,(k-6)\,(12)^{1/4}-\sqrt{\pi}\,\sqrt{\pi+a\,(k-6)\,2\,(12)^{1/4}}}{2\,(12)^{1/2}}\,,
\end{eqnarray}
Exploiting $a=3/50$ we find that if $k\ge 9$, then
\[
z_k\ge z_9=\frac{\pi+(9/50)\,(12)^{1/4}+\sqrt{\pi}\,\sqrt{\pi+(9/25)\,(12)^{1/4}}}{2\,(12)^{1/2}}>1\,,
\]
and since the last sign is strict (and $t_k^-<0$ for $k\ge 7$) we conclude that
\begin{equation}
  \label{pos1}
  \inf_{[0,1]}f_{k,1}(\cdot,0)\ge\inf_{[0,1]}f_{9,1}(\cdot,0)>0\,,\qquad\forall k\ge 9\,.
\end{equation}
For $8\ge k\ge 6$ we have
\begin{equation}
  \label{ya0}
  z_k\ge z_6=\frac{\pi}{(12)^{1/2}}>\frac9{10}\,,
\end{equation}
and the strict sign in \eqref{ya0}, $t_k^-<0$ if $k=7,8$, and $t_6^-=0$, gives
\begin{equation}
  \label{pos1.5}
  \min_{k=7,8}\,\inf_{[0,9/10]}f_{k,1}(\cdot,0)\ge\inf_{[0,9/10]}f_{7,1}(\cdot,0)>0\,,
\end{equation}
and prove \eqref{f61 1}, respectively. Finally, for the remaining cases $3\le k\le 5$, we see that $k\ge 3$ gives
\begin{eqnarray}\label{ya1}
z_k\ge z_3=\frac{\pi-(9/50)\,(12)^{1/4}+\sqrt{\pi}\,\sqrt{\pi-(9/25)\,(12)^{1/4}}}{2\,(12)^{1/2}}>\frac4{5}\,;
\end{eqnarray}
moreover, by differentiation, $t\in[2,6]\mapsto y_t$ is decreasing on $[2,6]$ and takes its minimum value at $k=6$ ($y_6=0$), so that, for $3\le k\le 5$, we have
\begin{eqnarray}\label{ya2}
  y_k\le y_3=\frac{\pi-(9/50)\,(12)^{1/4}-\sqrt{\pi}\,\sqrt{\pi-(9/25)\,(12)^{1/4}}}{2\,(12)^{1/2}}\,
  <\frac1{500}<\frac1{100}\,.
\end{eqnarray}
Thanks to the strict signs in \eqref{ya1} and \eqref{ya2} we find that
\begin{equation}
  \label{pos2}
  \inf_{3\le k\le 8}\inf_{[1/100,4/5]}f_{k,1}(x,0)>0\,.
\end{equation}
We conclude the proof of \eqref{fk1 1}, \eqref{fk1 1.5}, and \eqref{fk1 2} by setting
\[
c_1=\min\Big\{\inf_{[0,1]}f_{9,1}(\cdot,0),\inf_{[0,9/10]}f_{7,1}(\cdot,0),\min_{3\le k\le 5}\inf_{[1/100,4/5]}f_{k,1}(\cdot,0)\Big\}\,,
\]
where $c_1>0$ by \eqref{pos1}, \eqref{pos1.5}, and \eqref{pos2}.

\medskip

\noindent {\it Conclusion of the proof}: We conclude the proof of the theorem by proving \eqref{yokai2}, \eqref{yokai2 k6} and \eqref{yokai2 min5}. To begin with, we notice that \eqref{lower bound one}, \eqref{lower bound two} and \eqref{f61 1} imply \eqref{yokai2 k6}.

\medskip

Next, we prove \eqref{yokai2}. First of all, by \eqref{lower bound one}, \eqref{lower bound two}, and \eqref{fk1 1}, we find that
\begin{equation}
  \label{yokai2 k ge 9}
  \inf_{[0,1]\times\R}f_k\ge\min\{c_0,c_1\}>0\,,\qquad\forall k\ge 9\,.
\end{equation}
that is \eqref{yokai2} restricted to $k\ge 9$. Combining \eqref{yokai2 k ge 9} with  \eqref{lower bound one}, \eqref{lower bound two}, and \eqref{fk1 1.5}, we see that to complete the proof of \eqref{yokai2} we are left to show that
\begin{equation}
  \label{left1}
  \min_{k=7,8}\,\inf_{[9/10,1]\times[-5/2,4/5]}f_k>0\,.
\end{equation}
We begin the proof of \eqref{left1} by showing that
\begin{equation}
  \label{lb 78 1}
  \min_{k=7,8}\,\inf_{[9/10,1]\times[-5/2,-1/10]}f_k>0\,.
\end{equation}
The idea is obtaining a positive lower bound on $f_{k,2}$ all the way up to $y\le -1/10$ under the restriction $x\in[9/10,1]$ (recall that \eqref{lower bound two} holds only up to $y\le-5/2$, but on the wider range $x\in[0,1]$). To this end, let us recall that, if $y\le 0$, then
\[
f_k(x,y)\ge f_{k,2}(x,y)=2\,\sqrt{\pi\,(x+2\,|y|)}-(12)^{1/4}\,|y|+a\,(k-6)-2\,(12)^{1/4}\,x\,;
\]
now, $f_{k,2}(\cdot,y)$ ($y\le 0$) is decreasing on $[1/2,1]$ (and, thus, on $[9/10,1]$) since
\[
\frac{\pa f_{k,2}}{\pa x}(x,y)=\sqrt{\frac{\pi}{x+2\,|y|}}-2\,(12)^{1/4}\le\sqrt{\frac{\pi}{1/2}}-2\,(12)^{1/4}<-1\,,
\]
therefore, for all $y\le 0$ and $k\ge 7$,
\[
\inf_{[9/10,1]}f_k(\cdot,y)\ge f_{k,2}(1,y)\ge f_{7,2}(1,y)=2\,\sqrt{\pi\,(1+2\,|y|)}+\frac3{50}-(12)^{1/4}\,|y|-2\,(12)^{1/4}\,.
\]
Since $f_{7,2}(1,\cdot)$ is concave on $(-\infty,0]$ we find
\begin{eqnarray*}
\min_{k=7,8}\inf_{[9/10,1]\times[-5/2,-1/10]}f_k&\ge&\inf_{[-5/2,-1/10]}f_{7,2}(1,\cdot)
\\
&=&\min\Big\{f_{7,2}\Big(1,-\frac52\Big),f_{7,2}\Big(1,-\frac1{10}\Big)\Big\}\ge\min\Big\{\frac3{10},\frac3{100}\Big\}\,.
\end{eqnarray*}
(Notice that we cannot extend this lower bound to $[-5/2,0]$ since $f_{7,2}(1,0)<0$.) Having proved \eqref{lb 78 1}, we next show
\begin{equation}
  \label{lb 78 2}
  \min_{k=7,8}\,\inf_{[9/10,1]\times[1/10,4/5]}f_k>0\,.
\end{equation}
Working with $f_{k,1}$, and using the fact that $f_{k,1}(x,y)$ is increasing in $y$ and concave in $x$, we find that
\begin{eqnarray*}
\min_{k=7,8}\,\inf_{[9/10,1]\times[1/10,4/5]}f_k&\ge&\inf_{[9/10,1]}f_{7,1}\Big(\cdot,\frac1{10}\Big)
\\
&\ge&\min\Big\{f_{7,1}\Big(\frac9{10},\frac1{10}\Big),f_{7,1}\Big(1,\frac1{10}\Big)\Big\}\ge\min\Big\{\frac14,\frac3{50}\Big\}\,,
\end{eqnarray*}
thus proving \eqref{lb 78 2}. Thanks to \eqref{lb 78 1} and \eqref{lb 78 2}, in order to prove \eqref{left1}, and thus to conclude the proof of \eqref{yokai2}, we are left to show that
\begin{equation}
  \label{lb 78 3}
  \min_{k=7,8}\,\inf_{[9/10,1]\times[-1/10,1/10]}f_k>0\,.
\end{equation}
To this end we shall bound $f_k$ from below by using $f_{k,4}$. We first notice that since $x\ge 9/10$, $y\le 1/10$, $k\mapsto k\,\tan(\pi/k)$ is decreasing, and $k\ge 7$, we have
\[
\frac{\pa f_{k,4}}{\pa x}(x,y)=\sqrt{\frac{k\,\tan(\pi/k)}{x-y}}-2\,(12)^{1/4}\le
\sqrt{\frac{7\,\tan(\pi/7)}{(8/10)}}-2\,(12)^{1/4}< -1\,.
\]
Hence
\begin{equation}
  \label{hey}
  \min_{k=7,8}\,\inf_{[9/10,1]\times[-1/10,1/10]}f_k\ge \min_{k=7,8}\,\inf_{[-1/10,1/10]}f_{k,4}(1,\cdot)\,.
\end{equation}
Next we notice that
\[
f_{k,4}(1,y)=2\,\sqrt{k\,\tan\Big(\frac{\pi}{k}\Big)}\,\sqrt{1-y}+a\,(k-6)+(12)^{1/4}\,(y-2)\,,
\]
is concave on $y\in[-1/10,1/10]$, with
\begin{eqnarray*}
  &&f_{7,4}\Big(1,-\frac1{10}\Big)\ge \frac1{500}\,,\qquad f_{7,4}\Big(1,\frac1{10}\Big)\ge\frac{7}{1000}\,,
  \\
  &&f_{8,4}\Big(1,-\frac1{10}\Big)\ge \frac1{50}\,,\qquad f_{8,4}\Big(1,\frac1{10}\Big)\ge\frac{3}{100}\,,
\end{eqnarray*}
which, combined with \eqref{hey}, leads to \eqref{lb 78 3}. This completes the proof of \eqref{yokai2}.

\medskip

We finally prove \eqref{yokai2 min5}. Thanks to \eqref{lower bound one}, \eqref{lower bound two}, and \eqref{fk1 2}, it is enough to show that
\begin{equation}
  \label{hey3}
  \min_{3\le k\le 5}\,\,\inf_{[7/10,1]\times[-5/2,4/5]}\,f_k>0\,.
\end{equation}
To this end, we first show that
\begin{equation}
  \label{lb 34 1}
  \min_{3\le k\le 5}\,\,\,\inf_{[7/10,1]\times[-5/2,-3/10]}f_k>0\,.
\end{equation}
Indeed, as already noticed in the proof of \eqref{lb 78 1}, for every $y\le 0$, $f_{k,2}(\cdot,y)$ is decreasing on $[1/2,1]$. In particular, \[
\min_{3\le k\le 5}\,\,\,\inf_{[7/10,1]\times[-5/2,-3/10]}f_k\ge \inf_{[-5/2,-3/10]}f_{3,2}(1,\cdot)\,,
\]
where we have also used that $k\mapsto f_{k,2}(x,y)$ is increasing. Now,
\[
f_{3,2}(1,y)=2\,\sqrt{\pi\,(1+2\,|y|)}-(12)^{1/4}\,|y|-3\,a-2\,(12)^{1/4}\,,
\]
is concave on $y\in [-5/2,-3/10]$, with
\[
f_{3,2}\Big(1,-\frac52\Big)>\frac1{10}\,,\qquad f_{3,2}\Big(1,-\frac3{10}\Big)>\frac1{50}\,,
\]
so that \eqref{lb 34 1} follows. We next prove that
\begin{eqnarray}
  \label{lb 34 2}
  \min_{k=3,4}\,\,\inf_{[4/5,1]\times[-3/10,1/5]}f_k>0\,,
  \\
  \label{lb 5 2}
  \inf_{[4/5,1]\times[-3/10,19/100]}f_5>0\,,
\end{eqnarray}
To this end, we start noticing that when $3\le k\le 5$, $x\ge 4/5$, and $y\le 1/5$, we have
\[
\frac{\pa f_{k,4}}{\pa x}(x,y)=\sqrt{\frac{k\,\tan(\pi/k)}{x-y}}-2\,(12)^{1/4}
\le \sqrt{\frac{3\,\tan(\pi/3)}{3/5}}-2\,(12)^{1/4}\le-\frac45\,.
\]
In particular, using the concavity in $y$ of $f_{k,4}(1,y)$, we find
\begin{eqnarray*}
\min_{k=3,4}\inf_{[4/5,1]\times[-3/10,1/5]}f_k&\ge&\min_{k=3,4}\inf_{[-3/10,1/5]}f_{k,4}(1,\cdot)
\\
&=&\min_{k=3,4}\min\Big\{f_{k,4}\Big(1,-\frac3{10}\Big),f_{k,4}\Big(1,\frac1{5}\Big)\Big\}\,,
\end{eqnarray*}
where
\begin{eqnarray*}
  &&f_{3,4}\Big(1,-\frac3{10}\Big)\ge\frac{7}{10}\,,\qquad f_{3,4}\Big(1,\frac15\Big)\ge\frac12\,,
  \\
  && f_{4,4}\Big(1,-\frac3{10}\Big)\ge\frac1{10}\,,\qquad f_{4,4}\Big(1,\frac15\Big)\ge\frac1{10}\,,
\end{eqnarray*}
thus proving \eqref{lb 34 2}. In the case $k=5$ we would like to repeat the same argument but, unfortunately, $f_{5,4}(1,1/5)<0$. This is why, when $k=5$, we need to stop using the $f_{k,4}$ bound a bit below $y=1/5$. Stopping at $y=19/100$ works, and we find
\[
\inf_{[4/5,1]\times[-3/10,19/100]}f_5\ge\inf_{[-3/10,19/100]}f_{5,4}(1,y)
\ge\min\Big\{f_{5,4}\Big(1,-\frac3{10}\Big),f_{5,4}\Big(1,\frac{19}{100}\Big)\Big\}\,,
\]
where
\[
f_{5,4}\Big(1,-\frac3{10}\Big)\ge\frac{1}{200}\,,\qquad f_{5,4}\Big(1,\frac{19}{100}\Big)\ge\frac1{1000}\,,
\]
thus proving \eqref{lb 5 2}. Thanks to \eqref{lb 34 1}, \eqref{lb 34 2}, and \eqref{lb 5 2}, in order to deduce \eqref{hey3} (and thus complete the proof of the theorem), we are left to show that
\begin{eqnarray}
  \label{lb 34 3}
  \min_{k=3,4}\inf_{[4/5,1]\times[1/5,4/5]}f_k>0\,,
  \\
  \label{lb 5 3}
  \inf_{[4/5,1]\times[19/100,4/5]}f_5>0\,.
\end{eqnarray}
Using that $f_{k,1}$ is increasing in $y$ and concave in $x$ we find that
\begin{eqnarray*}
  \min_{k=3,4}\inf_{[4/5,1]\times[1/5,4/5]}f_k\ge \min_{k=3,4}\inf_{[4/5,1]}f_{k,1}\Big(\cdot,\frac15\Big)
  \ge\min_{k=3,4}\min\Big\{f_{k,1}\Big(\frac45,\frac15\Big),f_{k,1}\Big(1,\frac15\Big)\Big\}\,,
\end{eqnarray*}
where
\begin{eqnarray*}
  &&f_{3,1}\Big(\frac45,\frac15\Big)\ge\frac3{10}\,,\qquad   f_{3,1}\Big(1,\frac15\Big)\ge\frac1{100}\,,
  \\
  &&f_{4,1}\Big(\frac45,\frac15\Big)\ge\frac25\,,\qquad   f_{4,1}\Big(1,\frac15\Big)\ge\frac7{100}\,,
\end{eqnarray*}
thus proving \eqref{lb 34 3}. Similarly,
\begin{eqnarray*}
  \inf_{[4/5,1]\times[19/100,4/5]}f_5\ge \inf_{[4/5,1]}f_{5,1}\Big(\cdot,\frac{19}{100}\Big)
  \ge\min\Big\{f_{5,1}\Big(\frac45,\frac{19}{100}\Big),f_{5,1}\Big(1,\frac{19}{100}\Big)\Big\}\,,
\end{eqnarray*}
where
\[
f_{5,1}\Big(\frac45,\frac{19}{100}\Big)\ge\frac25 \,,\qquad f_{5,1}\Big(1,\frac{19}{100}\Big)\ge\frac1{10}\,,
\]
thus leading to \eqref{lb 5 3}. This completes the proof of the theorem.
\end{proof}

\appendix

\section{The isoperimetric inequality for immersed polygons}\label{Appendix B} In the proof of Theorem \ref{thm main}, and specifically in the proof of Theorem \ref{thm slight improve}, Section \ref{section hales intermed improv}, we have claimed that the isoperimetric inequality for $k$-gon \eqref{isoperimetric inequality kgons} is valid for {\it immersed} polygons (with the notion of oriented area used in place of the standard notion of area); see \eqref{see later}. Since we have not been able to find a proof of this more general inequality in the literature, we include one in this appendix.

\medskip

We identify a generic immersed polygon with $k$-edges with an ordered collection $\Pi_k$ of points $p_j\in\R^2$, $j=1,...,k$. Setting $p_j=(x_j,y_j)$ for the coordinates of these points, with $p_{k+1}=p_1$ the perimeter and oriented area of $\Pi_k$ are given by
\begin{eqnarray*}
&&P(\Pi_k)= \sum_{j=1}^k\sqrt{(x_j-x_{j+1})^2+(y_j-y_{j+1})^2}\,,
\\
&&A(\Pi_k) = \frac{1}{2}\sum_{j=1}^k (x_j\,y_{j+1}-x_{j+1}\,y_j)\,.
\end{eqnarray*}

\begin{theorem}[Isoperimetry for immersed $k$-gons]\label{lem: oriented polygonal isoperimetric} For every $\Pi_k$ as above we have
\[
P(\Pi_k)^2 \geq p(k)^2\,A(\Pi_k)\,,
\]
where $p(k)$ is the perimeter of a unit-area regular $k$-gon as defined in \eqref{p kappa}.
\end{theorem}

\begin{proof} After scaling, we can prove Theorem \ref{lem: oriented polygonal isoperimetric} by showing that the maximum of $A(\Pi_k)$ under the constraint $P(\Pi_k)=1$ is achieved when $\Pi_k$ is a regular $k$-gon with unit-perimeter. Clearly a maximum point $\Pi_k$ exists\footnote{A compactness argument shows the existence of a maximizer with at least $3$ and at most $k$ edges. To see that a maximizer has exactly $k$ edges one notices that, should this not be the case, area could be increased by first taking a ``triangular variation'' of one edge (thus adding one additional edge), and then by rescaling so to preserve the perimeter constraint.} and satisfies $A(\Pi_k)>0$. As shown in \cite{khimshpanina,leger}, the fact that $\Pi_k$ is a critical point of $A$ at $P$ fixed implies that all the vertexes of $\Pi_k$ belong to a circle or to a line, and since $A(\Pi_k)>0$, the first case holds. For some $R>0$ we can thus set $p_j = (R\,\cos\theta_j,R\,\sin\theta_j)$ for $1\le j\le k$. Setting
\[
x\equiv y\qquad\mbox{if and only if}\qquad x=y+2\,\pi\,h\quad\mbox{for some $h\in\mathbb{Z}$}\,,
\]
extending the definition of $\theta_j$ to every $j\in\mathbb{Z}$ by setting $\theta_{j+h\,k}=\theta_j$ for all $j=1,...,k$ and $h\in\mathbb{Z}$, and noticing that $\theta_{j+1}\not\equiv\theta_j$ for all $j$ (for, otherwise, $\Pi_k$ would have less than $k$ edges), we see that
\begin{eqnarray*}
A(\Pi_k) = \frac{R^2}{2}\,\sum_{j=1}^k\sin(\theta_{j+1}-\theta_j)\,,
\qquad P(\Pi_k) = \sqrt{2}\,R\,\sum_{j=1}^k\sqrt{1-\cos(\theta_{j+1}-\theta_j)}\,.
\end{eqnarray*}
Exploiting $P(\Pi_k)=1$ to solve for $R$, we come to define
\begin{eqnarray*}
&&a(\theta_1,...,\theta_k) := 4\,A(\Pi_k) = \frac{N}{D^2}\,,\qquad\mbox{where}
\\
&&\om_j=\theta_{j+1}-\theta_j\,,\qquad N=\sum_{j=1}^k\sin(\om_j)\,,\qquad D=\sum_{j=1}^k\sqrt{1-\cos(\om_j)}\,.
\end{eqnarray*}
We prove the theorem by showing that, at a maximum point of $a$, it holds $\om_j\equiv 2\,\pi/k$ for all $j$.

\medskip

\noindent {\it Step one}: We prove that
\begin{eqnarray}
\label{j odd}
&&\om_j\equiv\om_1\,,\qquad\mbox{if $j$ is odd}\,,
\\
\label{j even}
&&\om_j\equiv\om_2\,,\qquad\mbox{if $j$ is even}\,.
\end{eqnarray}
To this end, it will suffice to show
\begin{equation}
\label{theta separation}
\theta_j-\theta_{j-1}\equiv\theta_{j+2}-\theta_{j+1}\,,\qquad \forall j\,.
\end{equation}
We start by computing (recall that $\om_j\ne 0$ for all $j$)
\[
\frac{\pa a}{\pa\theta_j}=\frac{\cos(\om_{j-1})-\cos(\om_j)}{D^2}+\frac{N}{D^3}\,\Big\{\frac{\sin(\om_{j-1})}{\sqrt{1-\cos(\om_{j-1})}}-
\frac{\sin(\om_{j})}{\sqrt{1-\cos(\om_{j})}}\Big\}\,.
\]
Since $\cos(x)=\cos^2(x/2)-\sin^2(x/2)$ implies $\sqrt{1-\cos(x)}=\sqrt{2}\,|\sin(x/2)|$, using also $\sin(x)=2\,\sin(x/2)\,\cos(x/2)$, we find that
\[
\frac{\sin(\om_j)}{\sqrt{1-\cos(\om_j)}}=\frac{2\,\sin(\om_j/2)\,\cos(\om_j/2)}{\sqrt{2}\,|\sin(\om_j/2)|}
=\sqrt2\,\s_j\,\cos\Big(\frac{\om_j}2\Big)
\]
where we have set $\s_j={\rm sign}(\sin(\om_j/2))$. Hence,
\[
D^3\,\frac{\pa a}{\pa\theta_j}=
D\,\Big\{\cos(\om_{j-1})-\cos(\om_j)\Big\}
+\sqrt{2}\,N\,\Big\{\s_j\,\cos\Big(\frac{\om_j}2\Big)-\s_{j-1}\,\cos\Big(\frac{\om_{j-1}}2\Big)\Big\}\,.
\]
The condition that $a$ achieves its maximum at $(\theta_1,...,\theta_k)$ thus implies that
\begin{equation}
  \label{eqn: partial zero}
  \frac{\cos(\om_{j})-\cos(\om_{j-1})}
{\s_j\,\cos(\om_j/2)-\s_{j-1}\,\cos(\om_{j-1}/2)}
=\frac{\sqrt{2}\,N}D\,,\qquad\forall j=1,...,k\,,
\end{equation}
where, notably, the right-hand side is independent of $j$. By $\cos(x)=\cos^2(x/2)-\sin^2(x/2)=2\,\cos^2(x/2)-1$ and thanks to $\s_j^2=1$, we find that
\begin{eqnarray*}
  &&\cos(\om_{j})-\cos(\om_{j-1})=2\,\Big(\cos^2(\om_j/2)-\cos^2(\om_{j-1}/2)\Big)
  \\
  &&=
  2\,\Big(\s_j\,\cos(\om_j/2)-\s_{j-1}\,\cos(\om_{j-1}/2)\Big)\,\Big(\s_j\,\cos(\om_j/2)+\s_{j-1}\,\cos(\om_{j-1}/2)\Big)\,,
\end{eqnarray*}
which combined into \eqref{eqn: partial zero} gives
\[
\s_j\,\cos\Big(\frac{\om_j}2\Big)+\s_{j-1}\,\cos\Big(\frac{\om_{j-1}}2\Big)=\frac{N}{\sqrt{2}\,D}\,,\qquad\forall j=1,...,k\,.
\]
In particular, for all $j=1,...,k$,
\[
\s_{j+1}\,\cos\Big(\frac{\om_{j+1}}2\Big)+\s_j\,\cos\Big(\frac{\om_j}2\Big)=\s_j\,\cos\Big(\frac{\om_j}2\Big)+\s_{j-1}\,\cos\Big(\frac{\om_{j-1}}2\Big)\,,
\]
that is (setting $\theta_0=\theta_k$, $\om_0=\om_k$, and $\s_0=\s_k$)
\begin{equation}
  \label{eq: angle simplification}
  \s_{j+1}\,\cos\Big(\frac{\om_{j+1}}2\Big)=\s_{j-1}\,\cos\Big(\frac{\om_{j-1}}2\Big)\,,\qquad\forall j=1,...,k\,.
\end{equation}
Let us now denote by $t_j$ the unique element of $(-\pi,0)\cup(0,\pi)$ such that $t_j\equiv\om_j/2$, so that $\s_j={\rm sign}(\sin(t_j))$ and \eqref{eq: angle simplification} takes the form
\begin{equation}
  \label{star}
  \s_{j+1}\,\cos(t_{j+1})=\s_{j-1}\,\cos(t_{j-1})\,,\qquad\forall j=1,...,k\,.
\end{equation}
If $t_{j+1},t_{j-1}\in(0,\pi)$, then $\s_{j+1}=\s_{j-1}=1$, and $\cos(t_{j+1})=\cos(t_{j-1})$ implies $t_{j+1}=t_{j-1}$; similarly, if $t_{j+1},t_{j-1}\in(-\pi,0)$, then $\s_{j+1}=\s_{j-1}=-1$, and \eqref{star} gives again $t_{j+1}=t_{j-1}$; finally, if, without loss of generality, $t_{j+1}\in(0,\pi)$ and $t_{j-1}\in(-\pi,0)$, then \eqref{star} gives $\cos(t_{j+1})=-\cos(t_{j-1})$, which in turn implies $t_{j+1}=t_{j-1}+\pi$. Looking back at the definition of $t_j$, this proves \eqref{theta separation}.

\medskip

\noindent {\it Step two}: We prove that there is $n\in\{1,...,k-1\}$ such that, for all $j$,
\begin{equation}
  \label{step two appendix}
  \om_j\equiv 2\,\pi\,\frac{n}k\,,\qquad\forall j\,.
\end{equation}
We have separate arguments depending on whether $k$ is odd or even.

\medskip

\noindent {\it When $k$ is odd}, using first that $k+1$ is even in combination with \eqref{j even}, and then $\theta_{j+k}=\theta_j$ for all $j$, we find that
\[
\om_2\equiv\om_{k+1}=\theta_{k+2}-\theta_{k+1}=\theta_2-\theta_1=\om_1\,,
\]
that is, $\om_2\equiv\om_1$ and thus, by \eqref{j even} and \eqref{j odd}, $\om_j\equiv \om_1$ for all $j$. By repeatedly using this fact we find
\begin{equation}
  \label{in place of}
  \theta_1=\theta_{k+1}\equiv \om_1+\theta_k=\om_1+\sum_{j=2}^k(\theta_j-\theta_{j-1})+\theta_1\equiv k\,\om_1+\theta_1\,,
\end{equation}
that is, $\om_1 \equiv 2\,\pi\,(n/k)$ for some $n\in\{1,...,k\}$. Since $\om_j\not\equiv 0$ for all $j$, it cannot be that $n=k$, and the conclusion follows.

\medskip

\noindent {\it When $k$ is even}, using \eqref{j even} and \eqref{j odd}, in place of \eqref{in place of} we find
\begin{eqnarray}
  \theta_1=\theta_{k+1}=\theta_1+\sum_{j=2}^{k+1}(\theta_j-\theta_{j-1})\equiv \theta_1+\frac{k}2\,\om_2
  +\frac{k}2\,\om_1\,.
\end{eqnarray}
In particular, there exists $n\in\N$ such that
\begin{equation}
  \label{exploit}
  \om_1+\om_2=4\,\pi\, \frac{n}k\,.
\end{equation}
Notice that $n$ cannot be an integer multiple of $k$ or of $k/2$ (for, otherwise, $\om_2\equiv-\om_1$ would imply that all the points $p_j$ lie on a same line, thus implying $A(\Pi_k)=0$). Without loss of generality, we can thus assume that
\begin{equation}
  \label{exploit 2}
  \om_1+\om_2=4\,\tau\,,\quad \tau:=\pi\,\frac{n}k\,\qquad n=1,...,k-1\,,\,\,n\ne \frac{k}2\,.
\end{equation}
and use \eqref{j even}, \eqref{j odd}, and \eqref{exploit 2} to find that
\begin{eqnarray*}
a(\theta_1,...,\theta_k)&=&\frac{\sum_{j=1}^k\sin(\om_j)}{(\sum_{j=1}^k\sqrt{1-\cos(\om_j)})^2}
\\
&=&\frac{(k/2)\,(\sin\om_1+\sin\om_2)}{(k/2)^2\,(\sqrt{1-\cos\om_1}+\sqrt{1-\cos\om_2})^2}
\\
&=&\frac2k\,\frac{\sin(\om_1)+\sin(4\,\tau-\om_1))}{(\sqrt{1-\cos(\om_1)}+\sqrt{1-\cos(4\,\tau-\om_1)})^2}\,.
\end{eqnarray*}
Notice that, up to this point, we know that $\om_1\not\equiv 0$ (because, otherwise, we would have $p_1=p_2$) and that $\om_1\not\equiv4\,\tau$ (because, otherwise, \eqref{exploit 2} would give $\om_2\equiv 0$ and thus $p_3=p_2$). Setting $f(\om)=g(\om)/h(\om)^2$ for
\[
g(\om)=\sin(\om)+\sin(4\,\tau-\om)\,,\qquad h(\om)=\sqrt{1-\cos(\om)}+\sqrt{1-\cos(4\,\tau-\om)}\,,
\]
the fact that $\Pi_k$ maximizes $A$ under the constraint that $P=1$ implies that $\om=\om_1$ is a maximum point of the function $f(\om)$ over the open set $\Om=\{\om\in\R:\om\not\equiv0,4\,\tau\}$. Now, $f$ is differentiable in $\Om$, with
\begin{eqnarray*}
&&f'(\om)=\frac{\cos(\om)-\cos(4\,\tau-\om)}{h(\om)^2}-\frac{2\,g(\om)}{h(\om)^3}\,
\Big\{\frac{\sin(\om)}{2\,\sqrt{1-\cos(\om)}}
+\frac{-\sin(4\,\tau-\om)}{2\,\sqrt{1-\cos(4\,\tau-\om)}}\Big\}
\\
&&=\frac{\cos(\om)-\cos(4\,\tau-\om)}{h(\om)^2}-\frac{\sqrt2\,g(\om)}{h(\om)^3}\,
\Big\{\s\Big(\frac{\om}2\Big)\,\cos\Big(\frac{\om}2\Big)-\s\Big(2\,\tau-\frac{\om}2\Big)\,\cos\Big(2\,\tau-\frac{\om}2\Big)\Big\}\,,
\end{eqnarray*}
where we have used again $\sqrt{1-\cos(x)}=\sqrt2\,|\sin(x/2)|$ and $\sin(x)=2\,\sin(x/2)\,\cos(x/2)$, and where we have set $\s(\theta)={\rm sign}(\sin\theta)$. By $|\sin\theta|=\s(\theta)\,\sin\theta$ we find
\begin{eqnarray*}
\frac{h(\om)^3}{\sqrt2}\,f'(\om)\!\!&=&\!\!\Big(\cos(\om)-\cos(4\,\tau-\om)\Big)\,\Big\{\s\Big(\frac{\om}2\Big)\,\sin\Big(\frac{\om}2\Big)+
\s\Big(2\,\tau-\frac{\om}2\Big)\,\sin\Big(2\,\tau-\frac{\om}2\Big)\Big\}
\\
&&-\Big(\sin(\om)+\sin(4\,\tau-\om)\Big)\Big\{\s\Big(\frac{\om}2\Big)\,\cos\Big(\frac{\om}2\Big)-\s\Big(2\,\tau-\frac{\om}2\Big)\,
\cos\Big(2\,\tau-\frac{\om}2\Big)\Big\}
\\
\!\!&=&\!\!-2\,\sin(2\,\tau)\,\sin(\om-2\,\tau)\,\Big\{\s\Big(\frac{\om}2\Big)\,\sin\Big(\frac{\om}2\Big)+
\s\Big(2\,\tau-\frac{\om}2\Big)\,\sin\Big(2\,\tau-\frac{\om}2\Big)\Big\}
\\
&&-2\,\sin(2\,\tau)\,\cos(\om-2\,\tau)\Big\{\s\Big(\frac{\om}2\Big)\,\cos\Big(\frac{\om}2\Big)-\s\Big(2\,\tau-\frac{\om}2\Big)\,
\cos\Big(2\,\tau-\frac{\om}2\Big)\Big\}\,.
\end{eqnarray*}
The conditions on $\tau$ in \eqref{exploit 2} guarantee that $\sin(2\,\tau)\ne 0$, and thus allow us to infer the following identity from the fact that $\om_1$ is a maximum point of $f$ on the open set $\Om$,
\begin{eqnarray*}
  &&\s\Big(\frac{\om_1}2\Big)\,\Big\{
   \cos(\om_1-2\,\tau)\,\cos\Big(\frac{\om_1}2\Big)
  + \sin(\om_1-2\,\tau)\,\sin\Big(\frac{\om_1}2\Big)
  \Big\}
  \\
  &&=\s\Big(2\,\tau-\frac{\om_1}2\Big)\,\Big\{
   \cos(\om_1-2\,\tau)\,\cos\Big(2\,\tau-\frac{\om_1}2\Big)
  - \sin(\om_1-2\,\tau)\,\sin\Big(2\,\tau-\frac{\om_1}2\Big)
  \Big\}\,,
\end{eqnarray*}
that is, by $\cos(\a+\b)=\cos\a\,\cos\b-\sin\a\,\sin\b$,
\begin{eqnarray}\label{sigmas}
\s\Big(\frac{\om_1}2\Big)\,\cos\Big(2\tau-\frac{1}{2}\om_1\Big)
=\s\Big(2\,\tau-\frac{\om_1}2\Big)\,\cos\Big(\frac{1}{2}\om_1\Big)\,.
\end{eqnarray}
If both $\sigma$'s in \eqref{sigmas} have the same sign, then (by $\cos(x)=\cos(y)$ if and only if $x\equiv y$ or $x\equiv -y$) we either have
\[
2\tau-\frac{1}{2}\,\om_1 \equiv \frac{1}{2}\,\om_1\,,\qquad\mbox{i.e.}\quad \om_1\equiv 2\,\tau\,,
\]
which implies by \eqref{exploit 2} that $\omega_2 \equiv 2\tau$, and hence $\omega_1\equiv \omega_2$; or
\[
2\tau-\frac{1}{2}\,\om_1 \equiv -\frac{1}{2}\,\om_1\,,
\]
which leads to $2\,\tau\equiv 0$, in contradiction with \eqref{exploit 2}. If, instead, the $\sigma$'s in \eqref{sigmas} have opposite signs, recalling that $\cos(x)=-\cos(y)=\cos(y+\pi)$ if and only if $x\equiv y+\pi$ or $x\equiv -y+\pi$, we either have
\[
2\tau-\frac{1}{2}\,\om_1 \equiv \frac{1}{2}\,\om_1+\pi\,,\qquad\mbox{i.e.}\quad \om_1\equiv 2\,\tau-\pi\,,
\]
which implies by \eqref{exploit 2} that $\omega_2 \equiv 2\tau+\pi$, and hence also $\omega_1\equiv \omega_2$; or
\[
2\tau-\frac{1}{2}\,\om_1 \equiv -\frac{1}{2}\,\om_1+\pi\,,
\]
which leads to $2\,\tau\equiv\pi$, again in contradiction with \eqref{exploit 2}. We have thus proved that $\om_1\equiv\om_2$. Combining this fact with \eqref{exploit 2} gives \eqref{step two appendix}.

\medskip

\noindent {\it Conclusion}: We have proved so far that there is $n\in\{1,...,k-1\}$ such that $\om_j=\theta_{j+1}-\theta_j=2\,\pi\,n/k$ for all $j$. We are thus left to prove that $n=1$. To this end we notice that
\[
a(\theta_1,...,\theta_k)=\frac{\sum_{j=1}^k\sin(\om_j)}{(\sum_{j=1}^k\sqrt{1-\cos(\om_j)})^2}
=\frac1k\,\frac{\sin(2\,\pi\,n/k)}{1-\cos(2\,\pi\,n/k)}\,.
\]
Since $t\mapsto \sin(t)/(1-\cos(t))$ is decreasing for $t\in(0,2\,\pi)$, and since $\Pi_k$ is a maximum point of $A$ under $P=1$, we conclude that $n=1$.
\end{proof}

\bibliography{references}
\bibliographystyle{is-alpha}

\end{document}

%% file: figure.pstex_t
\begin{picture}(0,0)%
\includegraphics{figure.eps}%
\end{picture}%
\setlength{\unitlength}{3947sp}%
\begingroup\makeatletter\ifx\SetFigFont\undefined%
\gdef\SetFigFont#1#2#3#4#5{%
  \reset@font\fontsize{#1}{#2pt}%
  \fontfamily{#3}\fontseries{#4}\fontshape{#5}%
  \selectfont}%
\fi\endgroup%
\begin{picture}(4337,1080)(439,-691)
\put(1071,257){\makebox(0,0)[lb]{\smash{{\SetFigFont{8}{9.6}{\rmdefault}{\mddefault}{\updefault}{\color[rgb]{0,0,0}$2\,\pi/3$}%
}}}}
\end{picture}%

%% file: gamma.pstex_t
\begin{picture}(0,0)%
\includegraphics{gamma.eps}%
\end{picture}%
\setlength{\unitlength}{3947sp}%
\begingroup\makeatletter\ifx\SetFigFont\undefined%
\gdef\SetFigFont#1#2#3#4#5{%
  \reset@font\fontsize{#1}{#2pt}%
  \fontfamily{#3}\fontseries{#4}\fontshape{#5}%
  \selectfont}%
\fi\endgroup%
\begin{picture}(4037,1778)(442,-1093)
\put(3288,-843){\makebox(0,0)[lb]{\smash{{\SetFigFont{10}{12.0}{\rmdefault}{\mddefault}{\updefault}{\color[rgb]{0,0,0}$\g(s)$}%
}}}}
\put(4214,539){\makebox(0,0)[lb]{\smash{{\SetFigFont{10}{12.0}{\rmdefault}{\mddefault}{\updefault}{\color[rgb]{0,0,0}$\g|_{[s,t]}$}%
}}}}
\put(4464,208){\makebox(0,0)[lb]{\smash{{\SetFigFont{10}{12.0}{\rmdefault}{\mddefault}{\updefault}{\color[rgb]{0,0,0}$\g(t)$}%
}}}}
\put(3698,-484){\makebox(0,0)[lb]{\smash{{\SetFigFont{10}{12.0}{\rmdefault}{\mddefault}{\updefault}{\color[rgb]{0,0,0}$\ll\g(t),\g(s)\rr$}%
}}}}
\put(1855,-944){\makebox(0,0)[lb]{\smash{{\SetFigFont{10}{12.0}{\rmdefault}{\mddefault}{\updefault}{\color[rgb]{0,0,0}$s$}%
}}}}
\put(1868, 82){\makebox(0,0)[lb]{\smash{{\SetFigFont{10}{12.0}{\rmdefault}{\mddefault}{\updefault}{\color[rgb]{0,0,0}$t$}%
}}}}
\put(501,-450){\makebox(0,0)[lb]{\smash{{\SetFigFont{10}{12.0}{\rmdefault}{\mddefault}{\updefault}{\color[rgb]{0,0,0}$\mathbb{S}^1$}%
}}}}
\put(1985,-443){\makebox(0,0)[lb]{\smash{{\SetFigFont{10}{12.0}{\rmdefault}{\mddefault}{\updefault}{\color[rgb]{0,0,0}$[s,t]$}%
}}}}
\end{picture}%

%% file: arc.pstex_t
\begin{picture}(0,0)%
\includegraphics{arc.eps}%
\end{picture}%
\setlength{\unitlength}{3947sp}%
\begingroup\makeatletter\ifx\SetFigFont\undefined%
\gdef\SetFigFont#1#2#3#4#5{%
  \reset@font\fontsize{#1}{#2pt}%
  \fontfamily{#3}\fontseries{#4}\fontshape{#5}%
  \selectfont}%
\fi\endgroup%
\begin{picture}(4569,2346)(474,-2079)
\put(4370,-1419){\makebox(0,0)[lb]{\smash{{\SetFigFont{10}{12.0}{\rmdefault}{\mddefault}{\updefault}{\color[rgb]{0,0,0}$\theta$}%
}}}}
\put(4171,-72){\makebox(0,0)[lb]{\smash{{\SetFigFont{10}{12.0}{\rmdefault}{\mddefault}{\updefault}{\color[rgb]{0,0,0}$2\,\theta\,R$}%
}}}}
\put(3997,-2016){\makebox(0,0)[lb]{\smash{{\SetFigFont{10}{12.0}{\rmdefault}{\mddefault}{\updefault}{\color[rgb]{0,0,0}$2\,R\,\sin\theta$}%
}}}}
\put(4876,-1073){\makebox(0,0)[lb]{\smash{{\SetFigFont{10}{12.0}{\rmdefault}{\mddefault}{\updefault}{\color[rgb]{0,0,0}$R$}%
}}}}
\put(489,121){\makebox(0,0)[lb]{\smash{{\SetFigFont{10}{12.0}{\rmdefault}{\mddefault}{\updefault}{\color[rgb]{0,0,0}$(a)$}%
}}}}
\put(1045,-409){\makebox(0,0)[lb]{\smash{{\SetFigFont{10}{12.0}{\rmdefault}{\mddefault}{\updefault}{\color[rgb]{0,0,0}$\arc(\ell,x)$}%
}}}}
\put(3465,120){\makebox(0,0)[lb]{\smash{{\SetFigFont{10}{12.0}{\rmdefault}{\mddefault}{\updefault}{\color[rgb]{0,0,0}$(b)$}%
}}}}
\put(1251,-1840){\makebox(0,0)[lb]{\smash{{\SetFigFont{10}{12.0}{\rmdefault}{\mddefault}{\updefault}{\color[rgb]{0,0,0}$\ell$}%
}}}}
\put(1290,-1103){\makebox(0,0)[lb]{\smash{{\SetFigFont{10}{12.0}{\rmdefault}{\mddefault}{\updefault}{\color[rgb]{0,0,0}$x$}%
}}}}
\end{picture}%

%% file: alveare_final3.bbl
\newcommand{\etalchar}[1]{$^{#1}$}
\def\cprime{$'$}
\begin{thebibliography}{FAB{\etalchar{+}}93}

\bibitem[Alm76]{Almgren76}
F.~J.~Jr. Almgren.
\newblock Existence and regularity almost everywhere of solutions to elliptic
  variational problems with constraints.
\newblock {\em Mem. Amer. Math. Soc.}, 4\penalty0 (165):\penalty0 viii+199 pp,
  1976.

\bibitem[CLM16]{CLM1}
Marco Cicalese, Gian~Paolo Leonardi, and Francesco Maggi.
\newblock Improved convergence theorems for bubble clusters {I}. {T}he planar
  case.
\newblock {\em Indiana Univ. Math. J.}, 65\penalty0 (6):\penalty0 1979--2050,
  2016.

\bibitem[CM16]{carocciamaggi}
M.~Caroccia and F.~Maggi.
\newblock A sharp quantitative version of {H}ales' isoperimetric honeycomb
  theorem.
\newblock {\em J. Math. Pures Appl. (9)}, 106\penalty0 (5):\penalty0 935--956,
  2016.

\bibitem[FAB{\etalchar{+}}93]{foisyzimba}
Joel Foisy, Manuel Alfaro, Jeffrey Brock, Nickelous Hodges, and Jason Zimba.
\newblock The standard double soap bubble in {${\bf R}^2$} uniquely minimizes
  perimeter.
\newblock {\em Pacific J. Math.}, 159\penalty0 (1):\penalty0 47--59, 1993.

\bibitem[FT43]{fejes43}
L\'{a}szl\'{o} Fejes~Toth.
\newblock \"{U}ber das k\"{u}rzeste {K}urvennetz, das eine
  {K}ugeloberfl\"{a}che in fl\"{a}chengleiche konvexe {T}eile zerlegt.
\newblock {\em Math. Naturwiss. Anz. Ungar. Akad. Wiss.}, 62:\penalty0
  349--354, 1943.

\bibitem[Hal01]{hales}
T.~C. Hales.
\newblock The honeycomb conjecture.
\newblock {\em Discrete Comput. Geom.}, 25\penalty0 (1):\penalty0 1--22, 2001.

\bibitem[HM04]{heppesmorgan}
Aladar Heppes and Frank Morgan.
\newblock Planar clusters.
\newblock 2004.

\bibitem[IN15]{indreiPOLIGONI}
E.~Indrei and L.~Nurbekyan.
\newblock On the stability of the polygonal isoperimetric inequality.
\newblock {\em Adv. Math.}, 276:\penalty0 62?86, 2015.

\bibitem[KP08]{khimshpanina}
G.~Khimshiashvili and G.~Panina.
\newblock Cyclic polygons are critical points of area.
\newblock {\em Zapiski Nauchnykh Seminarov POMI}, 360:\penalty0 238--245, 2008.

\bibitem[Leg18]{leger}
J.C. Leger.
\newblock Aire, p\'{e}rim\`{e}tre et polygones cocycliques.
\newblock 2018.
\newblock arXiv: 1805.05423.

\bibitem[Mag12]{maggiBOOK}
F.~Maggi.
\newblock {\em Sets of finite perimeter and geometric variational problems},
  volume 135 of {\em Cambridge Studies in Advanced Mathematics}.
\newblock Cambridge University Press, Cambridge, 2012.
\newblock ISBN 978-1-107-02103-7.
\newblock xx+454 pp.
\newblock An introduction to {G}eometric {M}easure {T}heory.

\bibitem[Mor09]{Morgan}
F.~Morgan.
\newblock {\em Geometric measure theory. A beginner's guide. Fourth edition}.
\newblock Elsevier/Academic Press, Amsterdam, 2009.
\newblock viii+249 pp.

\bibitem[PT18]{paolinitamagnini}
Emanuele Paolini and Andrea Tamagnini.
\newblock Minimal clusters of four planar regions with the same area.
\newblock {\em ESAIM Control Optim. Calc. Var.}, 24\penalty0 (3):\penalty0
  1303--1331, 2018.

\bibitem[PT20]{paolinitortorelli}
E.~Paolini and V.~M. Tortorelli.
\newblock The quadruple planar bubble enclosing equal areas is symmetric.
\newblock {\em Calc. Var. Partial Differential Equations}, 59\penalty0
  (1):\penalty0 Paper No. 20, 9, 2020.

\bibitem[Wic04]{wichi}
W.~Wichiramala.
\newblock Proof of the planar triple bubble conjecture.
\newblock {\em J. Reine Angew. Math.}, 567:\penalty0 1--49, 2004.

\end{thebibliography}
